\documentclass[10pt]{article}
\usepackage{amsmath, amssymb ,amsthm, amsfonts, amsgen, color}
\usepackage{graphicx}
\usepackage{soul}
\usepackage{hyperref}

\numberwithin{equation}{section} 
\newcommand{\Nb}{{\mathbb{N}}}
\newcommand{\Rb}{{\mathbb{R}}}

\newcommand{\LL}{{\mathcal{L}}}

\makeatletter
\def\rightharpoonupfill@{\arrowfill@\relbar\relbar\rightharpoonup}
\newcommand{\xrightharpoonup}[2][]{\ext@arrow
0359\rightharpoonupfill@{#1}{#2}} \makeatother

\def\dist{\text{dist}}

\def\weakstar{\buildrel\star\over\rightharpoonup}
\def\eps{{\varepsilon}}

\def\weak{\rightharpoonup}

\newtheorem{theorem}{Theorem}[section]
\newtheorem{lemma}[theorem]{Lemma}
\newtheorem{proposition}[theorem]{Proposition}
\newtheorem{corollary}[theorem]{Corollary}
\newtheorem{remarks}[theorem]{Remarks}
\newtheorem{remark}[theorem]{Remark}
\newtheorem{definition}[theorem]{Definition}

\newcommand{{\rr}}{{\mathbb R}}

\newenvironment{@abssec}[1]{%
     \if@twocolumn
       \section*{#1}%
     \else
       \vspace{.05in}\footnotesize
       \parindent .2in
         {\upshape\bfseries #1. }\ignorespaces
     \fi}
     {\if@twocolumn\else\par\vspace{.1in}\fi}

\begin{document}

\title{\sc An optimal design problem with non-standard growth and no concentration effects }

\author{{\sc Ana Cristina Barroso}\\
Departamento de Matem\'atica and CMAFcIO \\ 
Faculdade de Ci\^encias da Universidade de Lisboa\\ 
Campo Grande, Edif\' \i cio C6, Piso 1\\
1749-016 Lisboa, Portugal\\
acbarroso@ciencias.ulisboa.pt\\
and \\
{\sc Elvira Zappale} \\
	Dipartimento di Science di Base ed Applicate per l'Ingegneria \\
	Sapienza - Universit\`a di Roma\\
	Via Antonio Scarpa, 16\\
	00161 Roma (RM), Italy\\
elvira.zappale@uniroma1.it}
\maketitle

\begin{abstract}
We obtain an integral representation for certain functionals arising in the context of optimal design and damage evolution problems under non-standard growth conditions and perimeter penalisation. Under our hypotheses, the integral representation includes a term
which is absolutely continuous with respect to the Lebesgue measure and a perimeter term, but no additional singular term.

We also study some dimension reduction problems providing results for the optimal design of thin films.

\end{abstract}

\smallskip



\section{Introduction}

In a recent article \cite{BZ}, we investigated the possibility of obtaining a measure representation, in a suitable sense (cf. Definition \ref{measrep}), for two functionals arising in certain relaxation processes for an energy of the type
\begin{equation}\label{F}
F\left(\chi,u\right)  :=\int_{\Omega}\chi\left(x\right) 
W_{1}\left(\nabla u(x)\right)  + \left(1-\chi\left(x\right)\right)  
W_{0}\left(\nabla u(x)\right) \, dx
+\left\vert D\chi\right\vert \left(\Omega\right),
\end{equation}
where $\Omega$ a bounded open subset of $\mathbb R^N$, $\chi \in BV\left(\Omega;\left\{  0,1\right\}\right)$ and 
$u \in W^{1,p}\left(\Omega;\mathbb{R}^{d}\right)$.

This energy has its origin in a problem in optimal design (see \cite{AL, AB, KL, KS1, KS2, KS3, MT}) where the perimeter term is added to ensure compactness, and thus existence, of solutions to the corresponding minimisation problems.
In this setting, the characteristic function $\chi$ corresponds either to one material, say $E \subset \Omega$, of a two components sample $\Omega$, or to one of the phases $E$ of a single material $\Omega$. The stored elastic energy or suitable function of the electrostatic potential density of $E$ is given by $W_1$, while $W_0$ is the energy associated to the other component or phase, and the term $|D\chi|(\Omega)$ penalises the measure of the created interfaces.

Another motivation comes from the modeling of ``brutal damage'', we refer to \cite{FM} where the first rigorous mathematical description was provided and to \cite{FF} for a nonlinear elastic setting in the framework of thin structures. Indeed, damage as an inelastic phenomenon can be described
by means of the characteristic function of the damaged
region which is a subset of $\Omega$, $\nabla u$ is the deformation strain, the elastic energy is given by the sum of the two contributions in the undamaged and damaged part, $W_0>W_1$, and a dissipational energy is taken as proportional, via the constant $\kappa > 0$ which represents the material toughness, to the damaged volume. This latter term corresponds to local cost of damaging a healthy part of the
sample. We refer to the recent paper \cite{BIR} and the bibliography therein for an asymptotic analysis, in the linear elastic case, where the damaged zones tend to disappear.

A regularisation term is added in the form of the total variation of the characteristic function $\chi$. Among the literature, we refer, for instance, to \cite{DMI,I,MR,PMM}, where a similar term is considered in the case where the damage parameter is assumed to range in the entire set $[0,1]$.   
Hence, the total energy contains an extra term with respect to \eqref{F}, and is given by
\begin{equation}\label{Fdam}
F_{d}\left(\chi,u\right)  :=\int_{\Omega}\chi\left(x\right) 
W_{1}\left(\nabla u(x)\right)  + \left(1-\chi\left(x\right)\right)  
W_{0}\left(\nabla u(x)\right) \, dx + \kappa \int_\Omega \chi(x) dx
+\left\vert D\chi\right\vert \left(\Omega\right).
\end{equation}
However, we observe that the extra dissipation term, being linear, does not add any particular difficulty to our analysis. Likewise, the possible addition of suitable boundary conditions or the work done by (linear) bulk loads pose no problems and thus are neglected in our subsequent description.

In the theory of shape optimisation, where the aim is to find an optimal shape minimising a cost functional (here the elastic energy), one should either impose 
directly a volume constraint on the phase where $\chi=1$ or, as in \eqref{Fdam}, the toughness $\kappa$ can be thought of as a Lagrange multiplier associated to a volume term. We also address these issues for the sake of completeness (see \eqref{vol}, Remark \ref{volfrac} below and the final comments after the proof of Proposition \ref{3D2DOgden}).

Letting $f: \{0,1\} \times \mathbb R^{d \times N} \to \mathbb R$
be defined as
\begin{equation}\label{density}
f\left(  b,\xi\right)  :=b  W_1(\xi)+ (1-b)W_0(\xi),
\end{equation}
to simplify the notation, the functionals considered in \cite{BZ} are given by
\begin{align}
\mathcal{F}\left(  \chi,u;A\right)   &  :=\inf\left\{  
\underset{n\rightarrow +\infty}{\lim\inf}\,
F\left(\chi_{n},u_{n};A\right)  :  u_{n}
\in W^{1,q}\left(A;\mathbb{R}^{d}\right),
\chi_{n}
\in BV\left(  A;\left\{  0,1\right\}  \right),  
\right. \label{introrelaxed}\\
&  \hspace{3cm} \left.  u_{n}\rightharpoonup u\text{ in }
W^{1,p}\left(  A;\mathbb{R}^{d}\right),
\chi_{n}\overset{\ast}{\rightharpoonup}\chi\text{ in }
BV\left(A;\left\{  0,1\right\}  \right)  \right\}\nonumber
\end{align}
and
\begin{align}
\mathcal{F}_{\operatorname*{loc}}\left(  \chi,u;A\right)   &  
:=\inf\left\{
\underset{n\rightarrow +\infty}{\lim\inf}\, F\left(  \chi_{n},u_{n};A\right)
: u_{n}  \in W_{\operatorname*{loc}}^{1,q}
\left(A;\mathbb{R}^{d}\right),
 \chi_{n}  \in 
BV\left(A;\left\{0,1\right\}  \right),  \right. \label{introrelaxedloc}\\
&  \hspace{3cm} \left.  u_{n}\rightharpoonup u\text{ in }W^{1,p}
\left(  A;\mathbb{R}^{d}\right), 
\chi_{n}\overset{\ast}{\rightharpoonup}\chi\text{ in }
BV\left(A;\left\{  0,1\right\}  \right)\right\}, \nonumber
\end{align}
where the exponents $p,q$ satisfy  
\begin{equation}\label{pqn}
1<p\leq q< \frac{N p}{N-1}
\end{equation} 
(if $N=1$ we let $1<p\leq q< +\infty$) and where we consider the localisation of \eqref{F} defined, for every open set $A \subset \Omega$ and every 
$(\chi ,u) \in BV(A;\{0,1\})\times W^{1,p}(A;\mathbb R^d)$, by
\begin{equation}
\label{FchiuA}
F(\chi, u; A):= \int_A f(\chi (x), \nabla u(x))\, dx + |D \chi|(A).
\end{equation}
The functions $W_{i}:\mathbb{R}^{d\times N}\rightarrow\mathbb{R}$, $i=0,1$, in
\eqref{density} are assumed to be continuous and satisfy the following growth condition
\begin{equation}\label{growth}
\exists \, \beta > 0 :
0 \leq W_{i}\left(  \xi\right)  \leq
\beta\left(  1+\left\vert \xi\right\vert ^{q}\right),
\; \; \; \forall \xi \in \Rb^{d \times N}.  
\end{equation}

Under the above hypotheses, in \cite{BZ} we showed that there exists a non-negative Radon measure $\mu$ defined
on the open subsets of $\overline{\Omega}$ which weakly represents $\mathcal{F}\left(  \chi,u;\cdot\right)$, whereas
$\mathcal{F}_{\operatorname*{loc}}\left(  \chi,u;\cdot\right)$ admits a strong measure representation (cf. Definition \ref{measrep}). Furthermore, assuming convexity of
$f(b,\cdot)$, $\forall b \in \{0,1\}$, we proved that, for every open subset $A$ of $\Omega$,
\begin{equation}\label{Flocrep}
{\cal F}_{\rm loc}(\chi, u;A) = \int_Af(\chi(x),\nabla u(x)) \, dx + |D \chi|(A)
+ \nu^s(\chi, u;A),
\end{equation}
where $\nu^s$ is a non-negative Radon measure, singular with respect to the Lebesgue measure (cf. Theorems 4.1 and 4.3 in \cite{BZ}).
This additional singular measure arises since in the above functionals
there is a gap between the space of admissible
macroscopic fields $u \in W^{1,p}\left(A;\mathbb{R}^{d}\right)$ and the smaller
space $W^{1,q}\left(A;\mathbb{R}^{d}\right)$ where the growth hypothesis \eqref{growth} ensures boundedness of the energy. Indeed, it is well known that
when no such gap is present, i.e. when $p=q$, and in the case independent of the field $\chi$, then
\begin{align*}
\mathcal{F}(u;A) = & \inf\left\{  \underset{n\rightarrow +\infty}{\liminf}
\int_{A}f(\nabla u_n(x)) \, dx  :  u_{n}
\in W^{1,q}\left(A;\mathbb{R}^{d}\right), 
 u_{n}\rightharpoonup u\text{ in }
W^{1,p}\left(  A;\mathbb{R}^{d}\right) \right\} \\
= & \int_{A}Qf(\nabla u(x)) \, dx, 
\end{align*}
where $Qf$ denotes the quasiconvex envelope of $f$ (see Definition \ref{qcxenv}).

For the range of exponents considered in \eqref{pqn}, similar functionals were
studied in \cite{ABF, BFM, FMy, MM}, the case where 
the integrability exponent $p(x)$ of the admissible fields depends in a continuous or regular piecewise continuous way on the location in the body was addressed in
\cite{CM, Mucci}, we also refer to \cite{HH} for generalisations of such problems in Orlicz type spaces.

In this paper, we expand on our previous results providing a full characterisation of 
$\mathcal{F}\left(  \chi,u;\Omega\right)$ and $\mathcal{F}_{\operatorname*{loc}}\left(  \chi,u;\Omega\right)$, under some hypotheses on $\chi$. We 
also 
assume that the continuous density functions 
$W_{i}:\mathbb{R}^{d\times N}\rightarrow\mathbb{R}$, $i=0,1$, in
\eqref{density} satisfy the 
following stronger 
growth condition
\begin{align}\label{growth2}
\exists \, \beta_1 > 0 : \; 
& 0 \leq W_{1}\left(  \xi\right)  \leq
\beta_1\left(  1+\left\vert \xi\right\vert ^{p}\right),
\; \; \; \forall \xi \in \Rb^{d \times N}, \\\label{growth3}
\exists \, \beta_0 > 0 : \; 
& 0 \leq W_{0}\left(  \xi\right)  \leq
\beta_0\left(  1+\left\vert \xi\right\vert ^{q}\right),
\; \; \; \forall \xi \in \Rb^{d \times N}, 
\end{align}
\noindent indeed \eqref{growth2} is a special case of \eqref{growth},
but we relax the condition \eqref{pqn} on the exponents $p$ and $q$, and we require just that 
\begin{equation}\label{newpqn}
1<p\leq q< +\infty.
\end{equation}
In the case under consideration, i.e. for suitably chosen $\chi$ and $u$, we show that the functionals \eqref{introrelaxed} and \eqref{introrelaxedloc} in question, 
evaluated at $\Omega$, admit an integral representation comprising a term which is absolutely continuous with respect to the Lebesgue measure, and a perimeter term, but there is no additional singular term.

Precisely, under certain structure assumptions on the fields 
$\chi \in BV(\Omega;\{0,1\})$ and, consequently, in view of \eqref{growth2} and \eqref{growth3}, also on $u \in W^{1,p}(\Omega;\mathbb R^d)$, we prove the following characterisation.

\begin{theorem}\label{main}
Let $\Omega \subset \Rb^N$ be a bounded, open extension domain.
Consider $p, q$ such that \eqref{newpqn} holds and let $f$ be defined as in 
\eqref{density}, satisfying \eqref{growth2}, \eqref{growth3} and
\begin{equation}\label{fconvex}
f(b,\cdot) \hbox{ is convex for every }b \in \{0,1\}.
\end{equation}
Assume $\chi$ is the characteristic function of an open, connected set of finite perimeter $E \subset \subset \Omega$ that satisfies 
$$
\mathcal H^{N-1}(\partial E) = P(E; \Rb^N), \; \;
\mathcal H^{N-1}(\partial (\Omega \setminus \overline E)) 
= P(\Omega \setminus \overline E; \Rb^N),$$
where $\partial E$ denotes the topological boundary of $E$ and $P(E; \Rb^N)$ is the perimeter of $E$ in $\Rb^N$,
and let $u \in W^{1,p}(\Omega;\mathbb R^d)$
be such that $u \in W^{1,q}(\Omega \setminus \overline{E};\mathbb R^d)$.

Then, 
\begin{equation}\label{rep}
{\cal F}_{\rm loc}(\chi, u;\Omega)= {\cal F}(\chi, u;\Omega) = \int_{\Omega}f(\chi(x),\nabla u(x)) \, dx + |D \chi|(\Omega) = F(\chi,u).
\end{equation}
\end{theorem}

In order to achieve these conclusions we do not invoke the results 
contained in either Theorem 4.1 or Theorem 4.3 in \cite{BZ}, which provide weak and strong measure representations for $\mathcal F$ and $\mathcal F_{\rm loc}$, respectively. Instead, we use a direct approach to prove double inequalities starting from \eqref{introrelaxed} and \eqref{introrelaxedloc} evaluated at $\Omega$.
For this reason, we allow for the less restrictive range of integrability exponents $p$ and $q$ considered in \eqref{newpqn}, as compared with \eqref{pqn}.

Although the problem under consideration is vectorial in nature, it is worth pointing out that in hypothesis \eqref{fconvex} we require convexity of $f(b,\cdot)$, rather than quasiconvexity. This is due to the fact that we have simultaneously an explicit dependence on the position in the body, through the field
$\chi$, and a gap problem. Indeed, a counterexample provided in \cite{ABF} shows that a representation of the form \eqref{rep} no longer holds when convexity is replaced with a weaker assumption.
However, it bears mentioning that even though we ask that $W_1$ be convex, in the case  of the density $W_0$ this hypothesis can be weakened,
see Remarks \ref{otherhyp}.

Our proof of Theorem \ref{main} is based on a result of Schmidt \cite{S}, which states that under some mild hypotheses on its boundary (see \eqref{goodset}), a set $E$ can be approximated from the inside by smooth sets, in such a way that the perimeters also converge
(cf. Theorem \ref{Schmidt}).

In particular, every set with Lipschitz boundary satisfies \eqref{goodset}.
However, if this condition fails to hold there are known counterexamples that show that the inner approximation by smooth sets may no longer be possible. We refer to Section 2 for more details.

We point out that one cannot expect the conclusions of Theorem \ref{main} to be true in general. Indeed, an example due to Zhikov \cite[page 467]{Zhikov}, and also considered in \cite[eq. (22) and example 1.15]{Mucci},
shows that some functionals where the integrand $f$ in \eqref{density}, has the form $f(b,\xi):= |\xi|^{bp + (1-b)q}$,
i.e., with a gap in the growth and coercivity exponents, do exhibit concentration effects. Taking $b=\chi_E$ fixed and $E$ the set defined by
\begin{equation}\label{E}
E:=\{ (x_1,x_2)\in B_1: x_1x_2>0\},
\end{equation} 
where $B_1\subset\subset  \Omega$ is the unit ball in $\mathbb R^2$, it is shown in
\cite{Zhikov} that if $W_0(\xi)= |\xi|^q$ and $W_1(\xi)= |\xi|^p$, with 
$1 < p < 2 < q$, then it is not possible to approximate in energy a target
$u \in W^{1,p}(\Omega; \Rb^d)$ using a more regular sequence
$u_n \in W^{1,q}(\Omega; \Rb^d)$. In our context, however, this example is ruled out
since $E$ lies in the class of Theorem \ref{Schmidt} and we are not required to work with a fixed $\chi_E$. The inner approximation result gives us
the freedom to approximate $\chi_E$ using a sequence of characteristic functions $\chi_{n}$ which allows us to create a buffer layer, separating the regions where $f$ has a different bound from above and $u$ has a different integrability exponent. It is, therefore, possible to construct a more regular sequence $u_n$ such that the energy $\displaystyle \liminf_{n \to +\infty}F(\chi_n, u_n)$ is bounded above by the energy $F(\chi,u)$ (cf. proof of Theorem \ref{main}).

However, in certain cases, energy concentrations do occur. Indeed, counterexamples obtained in \cite{MM} show that, unless a structure assumption is verified by the integrand, the relaxation process for the energy
$$
\inf\left\{  \underset{n\rightarrow +\infty}{\liminf}
\int_{A}f(x,\nabla u_n(x)) \, dx  :  u_{n}
\in W^{1,q}_{\rm loc}\left(A;\mathbb{R}^{d}\right), 
 u_{n}\rightharpoonup u\text{ in }
W^{1,p}\left(  A;\mathbb{R}^{d}\right) \right\}$$
leads to an infinite Dirac mass concentrated at a point. 

On the other hand, not every set of finite perimeter satisfies the assumptions imposed on the set $E$ in Theorem \ref{main}. Indeed, Example 3.53 in \cite{AFP} establishes the existence of an open set of finite perimeter $E$ in $\Rb^N$, $N \geq 2$, whose topological boundary has strictly positive Lebesgue measure and, thus, which fails to satisfy 
$\mathcal H^{N-1}(\partial E) = P(E; \Rb^N)$. This set, therefore, falls outside the scope of Theorem \ref{main}. Nevertheless (cf. \cite[Proposition 5.52]{AFP}), this situation does not hold in dimension $N=1$ which is in accordance with the fact that the one dimensional integral representation obtained in \cite[Proposition 4.8]{BZ} contains no singular measure.

Notice also that our integral representation result holds for pairs $(\chi,u)$ under the assumption that $u$ should be more regular in a certain subset of $\Omega$, namely $u \in W^{1,q}(\Omega \setminus \overline{E};\mathbb R^d)$, where 
$\chi$ is the characteristic function of the set $E$. However, this assumption is not too restrictive since if ${\cal F}(\chi, u;\Omega) < + \infty$, which is the usual condition considered in the literature (cf., for example, \cite{ABF,BFM,FMy}), and $W_0(\xi) = |\xi|^q$, then the additional regularity required of $u$ follows as a consequence, at least in the case where the set $E$ is sufficiently regular so that the Poincaré-Wirtinger inequality can be applied. 
	
Given the motivation stated above and the applications we have in mind, it will also be important
to consider the following related functional, where a volume constraint is imposed.
Given $0 < \theta < 1$ and 
$\chi \in BV\left(\Omega;\left\{  0,1\right\}  \right)$ such that
$\displaystyle \frac{1}{\LL^{N}(\Omega)}\int_\Omega \chi(x) \, dx =\theta$, we define
\begin{align}
\mathcal{F}_{\rm vol}\left(\chi,u;\Omega\right)   &  :=\inf\left\{  
\underset{n\rightarrow +\infty}{\lim\inf}\,
F\left(\chi_{n},u_{n};\Omega\right)  :  u_{n}
\in W^{1,q}\left(\Omega;\mathbb{R}^{d}\right),
\chi_{n}
\in BV\left(\Omega;\left\{  0,1\right\}  \right),  
\right. \label{vol}\\
&  \hspace{1cm} \left.  u_{n}\rightharpoonup u\text{ in }
W^{1,p}\left(\Omega;\mathbb{R}^{d}\right),
\chi_{n}\overset{\ast}{\rightharpoonup}\chi\text{ in }
BV\left(\Omega;\left\{  0,1\right\}  \right),  
\frac{1}{\LL^{N}(\Omega)}\int_\Omega \chi_n(x) \, dx =\theta \right\}.\nonumber
\end{align}
Under similar assumptions to those of Theorem \ref{main}, we show that \eqref{vol} admits the same integral representation as \eqref{rep} (cf. Remark \ref{volfrac}).

We organise the paper as follows. In Section 2 we set the notation 
and we provide some definitions and results
which will be used throughout. Theorem \ref{main} is proved in Section 3, where a similar cha\-rac\-te\-ri\-sa\-tion result is also shown in the case
where the convexity assumption
\eqref{fconvex} is replaced with the weaker assumption that the quasiconvex envelopes, $QW_1$ and $QW_0$ of $W_1$ and $W_0$, respectively, are convex.
Finally, in Section 4 we give some applications to dimension reduction problems. 

\section{Preliminaries}

In this section we fix notations and quote some definitions and 
results that will be used in the sequel.

Throughout the text $\Omega \subset \mathbb R^{N}$ will denote an open, bounded set.

We will use the following notations:
\begin{itemize}
\item ${\mathcal O}(\Omega)$ is the family of all open subsets
of $\Omega $;
\item $\mathcal M (\Omega)$ is the set of finite Radon
measures on $\Omega$;
\item $\left |\mu \right |$ stands for the total variation of a measure  $\mu\in \mathcal M (\Omega)$; 
\item $\mathcal L^{N}$ and $\mathcal H^{N-1}$ stand for the  $N$-dimensional Lebesgue measure 
and the $\left(  N-1\right)$-dimensional Hausdorff measure in $\mathbb R^N$, respectively;
\item the symbol $d x$ will also be used to denote integration with respect to $\mathcal L^{N}$;
\item  $C$ represents a generic positive constant that may change from line to line.
\end{itemize}

We start by recalling a well known result due to Ioffe 
\cite[Theorem 1.1]{Ioffe}.
\begin{theorem}\label{thm2.4ABFvariant}
Let $g:\mathbb R^m\times \mathbb R^{d \times N} \to [0, +\infty)$ 
be a Borel integrand such that $g(b,\cdot)$ is convex for every 
$b \in \mathbb R^m$. Then the functional 
$$
G(v,u):=\int_{\Omega}g(v(x), \nabla u(x)) \, dx
$$
is lower semicontinuous in 
$L^1(\Omega;\mathbb R^m)_{\rm strong} \times 
W^{1,1}(\Omega;\mathbb R^d)_{\rm weak}$. 
\end{theorem}

In the following we give some preliminary notions related with sets of finite
perimeter. For a detailed treatment we refer to \cite{AFP}.

To this end, we recall that a function $w\in L^{1}(\Omega;{\mathbb{R}}^{d})$ is said to be of
	\emph{bounded variation}, and we write $w\in BV(\Omega;{\mathbb{R}}^{d})$, if
	all its first order distributional derivatives $D_{j}w_{i}$ belong to $\mathcal{M}%
	(\Omega)$ for $1\leq i\leq d$ and $1\leq j\leq N$.

The matrix-valued measure whose entries are $D_{j}w_{i}$ is denoted by $Dw$
and $|Dw|$ stands for its total variation.
We observe that if $w\in BV(\Omega;\mathbb{R}^{d})$ then $w\mapsto|Dw|(\Omega)$ is lower
semicontinuous in $BV(\Omega;\mathbb{R}^{d})$ with respect to the
$L_{\mathrm{loc}}^{1}(\Omega;\mathbb{R}^{d})$ topology.

\begin{definition}
	\label{Setsoffiniteperimeter} Let $E$ be an $\mathcal{L}^{N}$- measurable
	subset of $\mathbb{R}^{N}$. For any open set $\Omega\subset\mathbb{R}^{N}$ the
	{\em perimeter} of $E$ in $\Omega$, denoted by $P(E;\Omega)$, is 
	given by
	\begin{equation}
	\label{perimeter}P(E;\Omega):=\sup\left\{  \int_{E} \mathrm{div}\varphi(x) \,dx:
	\varphi\in C^{1}_{c}(\Omega;\mathbb{R}^N), \|\varphi\|_{L^{\infty}}%
	\leq1\right\}  .
	\end{equation}
	We say that $E$ is a {\em set of finite perimeter} in $\Omega$ if $P(E;\Omega) <+
	\infty.$
\end{definition}

Recalling that if $\mathcal{L}^{N}(E \cap\Omega)$ is finite, then $\chi_{E}
\in L^{1}(\Omega)$, by \cite[Proposition 3.6]{AFP}, it follows
that $E$ has finite perimeter in $\Omega$ if and only if $\chi_{E} \in
BV(\Omega)$ and $P(E;\Omega)$ coincides with $|D\chi_{E}|(\Omega)$, the total
variation in $\Omega$ of the distributional derivative of $\chi_{E}$.
Moreover,  a
generalised Gauss-Green formula holds:
\begin{equation}\nonumber
{\int_{E}\mathrm{div}\varphi(x) \, dx
=\int_{\Omega}\left\langle\nu_{E}(x),\varphi(x)\right\rangle \, d|D\chi_{E}|,
\;\;\forall\,\varphi\in C_{c}^{1}(\Omega;\mathbb{R}^N)},
\end{equation}
where $D\chi_{E}=\nu_{E}|D\chi_{E}|$ is the polar decomposition of $D\chi_{E}$.

We also recall that, when dealing with sets of finite measure, a sequence of
sets $E_{n}$ converges to $E$ in measure in $\Omega$ if $\mathcal{L}%
^{N}(\Omega\cap(E_{n}\Delta E))$ converges to $0$ as $n\rightarrow +\infty$,
where $\Delta$ stands for the symmetric difference. 
This convergence is equivalent to $L^{1}(\Omega)$ 
convergence of the characteristic functions of the corresponding 
sets.

It is well known (cf. \cite{AFP}) that it is always possible to approximate, in measure, a set $E$ of finite perimeter in $\Rb^N$, with sets $E_{\eps}$ with smooth boundary, in such a way that the perimeters also converge. 
However, an open set of finite perimeter in $\Rb^N$ cannot, in general, be approximated strictly from within. In the sequel we rely on the following theorem due to Schmidt \cite{S} which states that, under mild hypotheses on its boundary, the approximation of the set $E$ is also true with the additional requirement
that the smooth sets satisfy $E_{\eps} \subset \subset E$.

\begin{theorem}\label{Schmidt}
[Strict interior approximation of the perimeter]. Let $E$ be
a bounded open set in $\Rb^N$ whose topological boundary $\partial E$ is well-behaved in the sense that 
\begin{equation}\label{goodset}
\mathcal H^{N-1}(\partial E) = P(E; \Rb^N).
\end{equation}
Then, for every $\eps > 0,$ there exists an open set $E_{\eps}$ with smooth boundary in $\Rb^N$ such that
\begin{equation}\label{sch}
E_{\eps} \subset \subset E, \; E \setminus E_{\eps} \subset N_{\eps}(\partial E) \cap N_{\eps}(\partial E_{\eps}), \; P(E_{\eps};\Rb^N) \leq P(E;\Rb^N) + \eps, 
\end{equation}
where we have used the notation $N_{\eps}(\cdot)$ for $\eps$-neighbourhoods of sets in $\Rb^N$.
\end{theorem}

The conditions \eqref{sch} imply, in particular, that 
$\displaystyle E = \bigcup_{\eps > 0}E_{\eps}$ and that 
\begin{equation}\nonumber
\lim_{\varepsilon \to 0^+}\LL^N(E_{\eps}) = \LL^N(E).
\end{equation}
On the other hand, the lower semicontinuity of the perimeter and the fact that
$\partial E_{\eps}$ are smooth, yield
\begin{equation}\nonumber\label{perconv}
\lim_{\varepsilon \to 0^+}\mathcal H^{N-1}(\partial E_{\eps}) 
= \lim_{\varepsilon \to 0^+} P(E_{\eps}; \Rb^N) = P(E; \Rb^N).
\end{equation}

In order to achieve the condition $P(E_{\eps};\Rb^N) \leq P(E;\Rb^N) + \eps$, rather
than the weaker bound $P(E_{\eps};\Rb^N) \leq C \mathcal H^{N-1}(\partial E)$, for
some constant $C > 1$, it is not sufficient to cover $\partial E$ with suitable balls
and to construct the approximants $E_{\eps}$ by removing these balls from $E$, but
instead a covering of $\partial E$ by suitably flat sets is required.

The conclusions of Theorem \ref{Schmidt} were already known to hold for bounded Lipschitz domains $E$ (see the references in \cite{S}). Indeed, every set with Lipschitz boundary satisfies \eqref{goodset}.
However, if this condition fails to hold there are known counterexamples that show that the inner approximation by smooth sets may no longer be possible. 
Indeed, letting 
$E = (0,1)^{N-1}\times \left((0,1) \setminus \{\frac{1}{2}\}\right)$,
and applying the lower semicontinuity of the perimeter on both halves of $E$, one
concludes that all approximations $E_{\eps}$ satisfy
$$\liminf_{\varepsilon \to 0^+}P(E_{\eps};\Rb^N) \geq 2N+2 > 2N = P(E;\Rb^N).$$
Notice that in this example
$$P(E;\Rb^N) = 2N < 2N+1 = \mathcal H^{N-1}(\partial E).$$
We also refer to Example 5.2 in \cite{S} and to Remark 1.27 in \cite{Giusti}.

We recall the notions of quasiconvex function and quasiconvex envelope which will be
used in Corollary \ref{Qconv}.

\begin{definition}\label{Morrey-qcx}
A Borel measurable and locally bounded function $f:\mathbb R^{d\times N}\to \mathbb R$ is said to be 
{\em quasiconvex} if
\begin{equation}\label{Mqcx}
f(\xi)\leq\frac{1}{{\mathcal L}^N(D)}\int_D 
f(\xi+\nabla \varphi(x)) \, dx,
\end{equation}
for every bounded, open set $D\subset \mathbb R^N$, for every $\xi \in \mathbb R^{d\times N}$ and for 
every $\varphi \in W^{1,\infty}_0(D;\mathbb R^d)$.  
\end{definition}

\begin{remark}\label{Mqcxobservation}
{\rm We recall that if \eqref{Mqcx} holds for a certain set $D$, then it holds for any bounded, open set 
in $\mathbb R^N$.
Notice also that, in the above definition, the value $+\infty$ is excluded from the range of $f$.}
\end{remark}

\begin{definition}\label{qcxenv}
The {\em quasiconvex envelope} 
of $f$ is the greatest
quasiconvex 
function that is less than or equal to $f$.
\end{definition}

We conclude this section by recalling the notions of weak and strong representation by means of measures.

\begin{definition}\label{measrep}
	Let $\mu$ be a Radon measure on $\overline{\Omega},$ let $(\chi, u) \in BV(\Omega;\{0,1\})\times W^{1,p}(\Omega;\mathbb R^d)$, and $\mathcal{G}\left(  \chi,u;\cdot\right)$ be a functional defined on $\mathcal O(\Omega)$. We say that
	
	\begin{enumerate}
		\item[a)] $\mu$ {\em (strongly) represents} 
		$\mathcal{G}\left(  \chi,u;\cdot\right)$ 
		if $\mu\left(  A\right)  =\mathcal{G}\left(  \chi,u;A\right)$ 
		for all open sets $A\subset\Omega;$
		
		\item[b)] $\mu~$ {\em weakly represents} 
		$\mathcal{G}\left(  \chi,u;\cdot\right)$
		if $\mu\left(  A\right)  \leq\mathcal{G}\left(\chi,u;A\right)  
		\leq \mu\left(  \overline{A}\right)$ for all open sets $A\subset\Omega.$
	\end{enumerate}
\end{definition}

\section{Main Result}
 
This section is devoted to the proof of Theorem \ref{main}. For the readers' convenience, we restate it here.

\medskip
\noindent{\bf Theorem 1.1}
{\it Let $\Omega \subset \Rb^N$ be a bounded, open extension domain.
Consider $p, q$ such that \eqref{newpqn} holds, let $f$ be defined as in 
\eqref{density}, satisfying \eqref{growth2}, \eqref{growth3} and \eqref{fconvex}.
Assume $\chi$ is the characteristic function of an open, connected set of finite perimeter $E \subset \subset \Omega$ such that $E$ and $\Omega \setminus \overline E$ satisfy \eqref{goodset} and let $u \in W^{1,p}(\Omega;\mathbb R^d)$
be such that $u \in W^{1,q}(\Omega \setminus \overline E;\mathbb R^d)$.
Then, 
\begin{equation}\nonumber
{\cal F}_{\rm loc}(\chi, u;\Omega)= {\cal F}(\chi, u;\Omega) = \int_{\Omega}f(\chi(x),\nabla u(x)) \, dx + |D \chi|(\Omega)  = F(\chi,u).
\end{equation}}

\begin{remarks}{\rm 
Hypothesis \eqref{growth2}, \eqref{growth3}, and the requirements placed on $u$, ensure that $F(\chi,u) < + \infty$ and so, from the upper bound inequality proved below, it follows that
$\mathcal{F}_{\operatorname*{loc}}\left(  \chi,u;\Omega\right) < + \infty$.

On the other hand, the conditions $E \in \mathcal{O}(\Omega), E \subset \subset \Omega$ and \eqref{goodset}, yield
$$\mathcal H^{N-1}(\partial E) = P(E; \Omega) = P(E; \Rb^N) = |D\chi|(\Omega) = 
|D\chi|(\Rb^N).$$
We also point out that, given the nature of the problem (see \eqref{F}, \eqref{growth2} and \eqref{growth3}), the assumptions made on $u$ depend on the set $E$ so, in the above integral representation result, the fields $\chi$ and $u$ are not independent of each other.

Since for the proof of the upper bound we use recovery sequences which are in $W^{1,q}(\Omega;\mathbb R^d)$, we show both that $\mathcal F(\chi, u;\Omega)$ admits an integral representation, and that it coincides with 
$\mathcal F_{\rm loc}(\chi, u;\Omega)$ for which $\nu^s$ in \eqref{Flocrep} vanishes.}

\end{remarks}

\begin{proof}[{\it Proof of Theorem \ref{main}}]
We obtain the characterisation of ${\cal F}_{\rm loc}(\chi, u;\Omega)$ and
${\cal F}(\chi, u;\Omega)$ directly, by proving a double inequality.

Due to the convexity hypothesis, the proof of the lower bound follows as in the second part of the proof of Theorem 4.3 in \cite{BZ} by means of Theorem \ref{thm2.4ABFvariant}.

To prove the upper bound we need to construct sequences 
$u_n \in W^{1,q} \left(\Omega;\mathbb{R}^{d}\right)$,
$\chi_{n}  \in BV\left(\Omega;\left\{0,1\right\}  \right)$ such that $u_{n}\rightharpoonup u\text{ in }W^{1,p}\left(\Omega;\mathbb{R}^{d}\right)$, 
$\chi_{n}\overset{\ast}{\rightharpoonup}\chi\text{ in }
BV\left(\Omega;\left\{  0,1\right\}  \right)$ and
$\displaystyle \liminf_{n\to+\infty}F(\chi_n,u_n) \leq  F(\chi,u)$.

Given that $\Omega$ is an extension domain and that 
$u \in W^{1,q}(\Omega \setminus \overline E; \Rb^d)$ we can extend $u$ as a $W^{1,q}$ function to $\Rb^N \setminus \Omega$, with an abuse of notation we still denote this function by $u$. The hypothesis on the set $E$ allows us to apply Theorem \ref{Schmidt} twice to obtain, for each $\eps > 0$, open sets 
$E_{2\eps} \subset \subset E_{\eps} \subset \subset E$ such that
$\partial E_{2\eps}$ and $\partial E_{\eps}$ are smooth,
\begin{equation}\label{per}
\lim_{\varepsilon \to 0^+}P(E_{2\eps};\Omega) = 
\lim_{\varepsilon \to 0^+}P(E_{\eps};\Omega) = P(E;\Omega),
\end{equation}
and 
\begin{equation}\label{meas}
\lim_{\varepsilon \to 0^+}\LL^N(E_{2\eps}) = 
\lim_{\varepsilon \to 0^+}\LL^N(E_{\eps}) = \LL^N(E).
\end{equation}
Denote by $L_\eps$ the open layer between $E_\eps$ and $E_{2\eps}$, 
$L_\eps := E_\eps \setminus \overline{E_{2\eps}}$, so that, by \eqref{meas},
\begin{equation}\label{measL}
\lim_{\varepsilon \to 0^+}\LL^N(L_{\eps}) = 0. 
\end{equation}
We may also consider $\Omega_\eps \subset \subset \Omega \setminus \overline E$ an inner approximation of $\Omega \setminus \overline E$ such that 
$\partial \Omega_{\eps}$ is smooth and 
\begin{equation}\label{outerlayer}
\lim_{\varepsilon \to 0^+}P(\Omega_{\eps};\Rb^N) 
= P(\Omega \setminus \overline E;\Rb^N), \; \; 
\lim_{\varepsilon \to 0^+}\LL^N(\Omega_{\eps}) = \LL^N(\Omega \setminus \overline E).
\end{equation}
Denote by $F_\eps$ the layer between $\Omega_\eps$ and $E$, notice that conditions \eqref{outerlayer} imply that
\begin{equation}\label{Feps}
\lim_{\varepsilon \to 0^+}P(F_{\eps};\Omega) = P(E;\Omega), \; \; 
\lim_{\varepsilon \to 0^+}\LL^N(F_{\eps}) = 0.
\end{equation}
This inner approximation construction is important to ensure the existence of the layer $L_\eps$, where we will connect two different regular sequences both converging to $u$ in $L^p$, and of the layer $F_\eps$, separating the regions where $f$ has different growth from above and where $u$ has different integrability properties.

Let $\{\rho_j\}_{j \in \Nb}$ be the usual sequence of standard mollifiers and, for each $\eps > 0$, consider the convolutions given by 
$u_{\eps,j} := u\cdot \chi_{E_{\eps}} \star \rho_j$  
and $\widetilde{u}_{\varepsilon,j}
:= u \cdot \chi_{\Omega \setminus \overline E_{2 \eps}}\star \rho_j$.
Notice that both ${u}_{\varepsilon,j}$ and $\widetilde{u}_{\varepsilon,j}$ converge
strongly to $u$ in $W^{1,p}(L_\eps;\mathbb R^d)$ and that, for $j$ large enough and
$ x \in \Omega \setminus (E \cup F_\eps)$, 
$B(x,\frac{1}{j}) \subset \Omega \setminus \overline E$ where $u \in W^{1,q}$ so that
$\widetilde{u}_{\varepsilon,j}$ converges, as $j \to + \infty$, to $u$ in  
$W^{1,q}(\Omega \setminus (\overline{E \cup F_\eps});\mathbb R^d)$.

Partition $L_{\eps}$ into $T_{\eps,j}$ $\in \mathbb N$ pairwise disjoint layers 
$$
L^{i}_{\eps,j}
:=\big\{x\in L_{\eps}:\,  \delta_{i-1}<\dist\big(x, \partial E_\eps  \big)\leq \delta_{i}\big\},
\qquad i=1, \dots, T_{\eps,j},
$$
of constant width 
$\delta_i-\delta_{i-1}
=\|u_{\eps,j}- \widetilde{u}_{\eps,j}\|_{L^p(L_{\eps};\Rb^d)}^{1/p}$,
with $\delta_0=0$ and $\delta_{T_{\eps,j}}=O(\eps)$, 
so that
\begin{equation}\label{Mlayers}
T_{\eps,j} \|u_{\eps,j}- \widetilde{u}_{\eps,j}\|_{L^p(L_{\eps};\Rb^d)}^{1/p}
=O(\eps),
\end{equation}
and, since $u_{\varepsilon, j}, \widetilde{u}_{\varepsilon, j} \in 
W^{1,p}(\Omega;\mathbb R^d)$,
\begin{align*}
&\displaystyle\sum_{i=1}^{T_{\eps,j}} \int_{L^{i}_{\eps,j}} 
\bigg(1 + \|\nabla u_{\eps,j}(x)-\nabla \widetilde{u}_{\eps,j}(x)\|^p 
+ \|\nabla \widetilde{u}_{\eps,j}(x)\|^p + \frac{|u_{\eps,j}(x)-\widetilde{u}_{\eps,j}(x)|^p}
{\|u_{\eps,j}-\widetilde{u}_{\eps,j}\|_{L^p(L_{\eps};\Rb^d)} } \bigg)\, dx \\
&\hspace{0.5cm}=\int_{L_{\eps}} 
\bigg(1 + \|\nabla u_{\eps,j}(x)-\nabla \widetilde{u}_{\eps,j}(x)\|^p 
+ \|\nabla \widetilde{u}_{\eps,j}(x)\|^p + \frac{|u_{\eps,j}(x)-\widetilde{u}_{\eps,j}(x)|^p}
{\|u_{\eps,j}-\widetilde{u}_{\eps,j}\|_{L^p(L_{\eps};\Rb^d)} } \bigg)\, dx.
\end{align*}
Thus, there exists $i_*=i_*(\eps,j)$ such that
\begin{align}\nonumber
&\int_{L^{i_*}_{\eps,j}} 
\bigg(1 + \|\nabla u_{\eps,j}(x)-\nabla \widetilde{u}_{\eps,j}(x)\|^p 
+ \|\nabla \widetilde{u}_{\eps,j}(x)\|^p + \frac{|u_{\eps,j}(x)-\widetilde{u}_{\eps,j}(x)|^p}
{\|u_{\eps,j}-\widetilde{u}_{\eps,j}\|_{L^p(L_{\eps};\Rb^d)} } \bigg)\, dx \\
&\hspace{0.5cm} \leq \frac{1}{T_{\eps,j}} \int_{L_{\eps}} 
\bigg(1 + \|\nabla u_{\eps,j}(x)-\nabla \widetilde{u}_{\eps,j}(x)\|^p 
+ \|\nabla \widetilde{u}_{\eps,j}(x)\|^p + \frac{|u_{\eps,j}(x)-\widetilde{u}_{\eps,j}(x)|^p}
{\|u_{\eps,j}-\widetilde{u}_{\eps,j}\|_{L^p(L_{\eps};\Rb^d)} } \bigg)\, dx.
\label{istar}
\end{align}

Consider cut-off functions $\varphi_{\eps,j}\in C^\infty_c(\Omega; [0,1])$ such that
\begin{equation*}\label{ncut-offc}\begin{split}
& \varphi_{\eps,j}=1 \qquad \text{ in}\; \;  E_{2\eps} \cup 
\Bigg(\bigcup_{i=i_{*} + 1}^{T_{\eps,j}} L^{i}_{\eps,j}\Bigg)=:  A_{\eps,j},\\
& \varphi_{\eps,j}=0 \qquad \text{ in}\; \;
\Omega \setminus \Big(A_{\eps,j} \cup  L^{i_{*}}_{\eps,j}\Big),
\vspace{.2truecm}
\end{split}\end{equation*}
and 
\begin{equation}\label{grad_cutoff}
\| \nabla \varphi_{\eps,j}\|_\infty
=O\Big(\|u_{\eps,j} - \widetilde{u}_{\eps,j}\|_{L^p(L_{\eps};\Rb^d)}^{-1/p}\Big).
\end{equation}
Define
\begin{equation}\label{wnova}
w_{\eps,j}(x):= \varphi_{\eps,j}(x) u_{\eps,j}(x) +(1-\varphi_{\eps,j}(x))\widetilde{u}_{\eps,j}(x),
\; \; x\in \Omega,
\end{equation}   
and notice that 
\begin{equation}\label{nablaw}
\nabla w_{\eps,j} 
= \varphi_{\eps,j} (\nabla u_{\eps,j} -\nabla \widetilde{u}_{\eps,j}) 
+ \nabla \widetilde{u}_{\eps,j} + (u_{\eps,j} - \widetilde{u}_{\eps,j}) \otimes \nabla \varphi_{\eps,j} \; \; {\rm in} \; \;  L^{i_*}_{\eps,j}.
\end{equation} 
Then
$w_{\eps,j}\in  W^{1,q}(\Omega;\Rb^d)$,
and by properties of the convolution we have that, 
for each $\eps > 0$,
\begin{align*}
\lim_{j\to+\infty} \|  w_{\eps,j}-u\|_{L^p(\Omega;\Rb^d)}^p \leq
\lim_{j\to+\infty}\left[\int_{A_{\eps,j}}| u_{\eps,j}(x) - u(x) |^p \, dx + 
\int_{L^{i_*}_{\eps,j}}| u_{\eps,j}(x) - u(x) |^p \, dx  \right.\\
\left.+ \int_{L^{i_*}_{\eps,j}}| \widetilde{u}_{\eps,j}(x) - u(x) |^p \, dx +\int_{\Omega \setminus (A_{\eps,j} \cup L^{i_*}_{\eps,j})}| \widetilde{u}_{\eps,j}(x) - u(x) |^p \, dx \right]= 0,
\end{align*}
since $0 \leq \varphi_{\eps,j} \leq 1$ and using the fact that, as $j \to +\infty$,
$u_{\eps,j}$ converges to $u$ in $L^p(E_\eps;\Rb^d)$ and  
$\widetilde{u}_{\eps,j}$ converges to $u$ in 
$L^p(\Omega \setminus \overline{E_{2\eps}};\Rb^d)$.
Furthermore, taking the sequence $\chi_{\eps}$ to be the characteristic function of the set $E \cup F_\eps$, it follows that
\begin{align}\label{decomp}
& \nonumber
\int_{\Omega} \chi_{\eps}(x)\, W_1(\nabla w_{\eps,j}(x)) + 
(1 -\chi_{\eps}(x))\, W_0(\nabla w_{\eps,j}(x)) \, dx + |D\chi_{\eps}|(\Omega)
\\ \nonumber
&\qquad = \int_{A_{\eps,j}} W_1(\nabla u_{\eps,j}(x))\, dx +
\int_{\bigcup_{i=1}^{i_* - 1}L^i_{\eps,j} \cup (E\setminus E_\eps) \cup F_\eps} 
W_1(\nabla \widetilde{u}_{\eps,j}(x))\, dx + 
\int_{L^{i_*}_{\eps,j}} W_1(\nabla w_{\eps,j}(x))\, dx \\ \nonumber
&\qquad \qquad  
+ \int_{\Omega \setminus (E \cup F_\eps)} W_0(\nabla \widetilde{u}_{\eps,j}(x))\, dx
+|D\chi_\eps|(\Omega)
\\ \nonumber
&\qquad \leq 
\int_{A_{\eps,j}}  W_1(\nabla u_{\eps,j}(x))\, dx + 
\int_{L_\eps \cup (E \setminus E_{\eps}) \cup F_\eps} W_1(\nabla \widetilde{u}_{\eps,j}(x))\, dx + \int_{L^{i_*}_{\eps,j}} W_1(\nabla w_{\eps,j}(x))\, dx
\\ 
&\qquad \qquad + 
\int_{\Omega \setminus (E \cup F_\eps)} W_0(\nabla \widetilde{u}_{\eps,j}(x))\, dx  + |D\chi_\eps|(\Omega).
\end{align}
By \eqref{growth2}, \eqref{meas}, \eqref{measL}, \eqref{Feps}, and since 
$\displaystyle \lim_{j\to +\infty}\|\widetilde{u}_{\eps,j} - u\|_{W^{1,p}(\Omega \setminus \overline{E_{2\eps}}) }= 0$,
\begin{align}\label{W10}
&\lim_{\varepsilon \to 0^+}\lim_{j\to+\infty}
\int_{L_{\eps} \cup (E\setminus E_\eps)\cup F_\eps}  W_1(\nabla \widetilde{u}_{\eps,j}(x))\, dx \leq \lim_{\varepsilon \to 0^+}\lim_{j\to+\infty} 
C \int_{L_{\eps} \cup (E\setminus E_\eps)\cup F_\eps} 
1 + \|\nabla \widetilde{u}_{\eps,j}(x)\|^p \, dx 
\nonumber \\ 
&\hspace{2cm}\leq \lim_{\varepsilon \to 0^+}\lim_{j\to+\infty} 
C \int_{L_{\eps} \cup (E\setminus E_\eps)\cup F_\eps} 
1 + \|\nabla \widetilde{u}_{\eps,j}(x) - \nabla u(x)\|^p + \|\nabla u(x)\|^p \, dx 
= 0.
\end{align}
By \eqref{Feps}
\begin{equation}\label{nconvper}
\lim_{\varepsilon \to 0^+}|D\chi_{\eps}|(\Omega) = |D\chi|(\Omega).
\end{equation}
On the other hand, by \eqref{growth2}, \eqref{nablaw}, \eqref{grad_cutoff}, \eqref{istar} and \eqref{Mlayers}, we have
\begin{align}
& \liminf_{j\to+\infty}\int_{L^{i_*}_{\eps,j}} W_1(\nabla w_{\eps,j}(x))\, dx
\leq \liminf_{j\to+\infty}\beta_1\int_{L^{i_*}_{\eps,j}} 
\big(1 + \|\nabla w_{\eps,j}(x)\|^p\big)\, dx \nonumber \\ \nonumber
& \leq C\limsup_{j\to+\infty}\int_{L^{i_*}_{\eps,j}} \hspace{-0,17cm}
\bigg(1 + |\varphi_{\eps,j}|^p\|\nabla u_{\eps,j}-\nabla \widetilde{u}_{\eps,j}\|^p 
+ \|\nabla \widetilde{u}_{\eps,j}\|^p + |u_{\eps,j}-\widetilde{u}_{\eps,j}|^p \cdot \|\nabla \varphi_{\eps,j}\|^p \bigg) dx \\ \nonumber
& \leq \limsup_{j\to+\infty}\frac{C}{T_{\eps,j}}\int_{L_{\eps}} \hspace{-0,15cm}
\bigg(1 + \|\nabla u_{\eps,j}(x)-\nabla \widetilde{u}_{\eps,j}(x)\|^p 
+ \|\nabla \widetilde{u}_{\eps,j}(x)\|^p + \frac{|u_{\eps,j}(x)-\widetilde{u}_{\eps,j}(x)|^p}{\|u_{\eps,j}-\widetilde{u}_{\eps,j}\|_{L^p(L_{\eps};\Rb^d)} } \bigg) \, dx 
\\ \nonumber
& \leq 
\limsup_{j\to+\infty}
\frac{C}{\eps}
\|u_{\eps,j}-\widetilde{u}_{\eps,j}\|^{\frac{1}{p}}_{L^p(L_{\eps};\Rb^d)}
\bigg[\int_{L_{\eps}} \hspace{-0,15cm}
\bigg(1 + \|\nabla u_{\eps,j}-\nabla \widetilde{u}_{\eps,j}\|^p 
+ \|\nabla \widetilde{u}_{\eps,j}\|^p \bigg)\, dx + \|u_{\eps,j}-\widetilde{u}_{\eps,j}\|^{p-1}_{L^p(L_{\eps};\Rb^d)}\bigg]
\\ \label{smallstrip}
& = 0,  
\end{align}
since the sequences $u_{\eps,j}$, $\widetilde{u}_{\eps,j}$ converge strongly to $u$ in $L^p(L_{\eps};\Rb^d)$, as $j\to+\infty$, and the expression in square brackets is uniformly bounded in $j$.

Finally, using the convexity of $W_0$ and $W_1$, and since the measures $\mu^j_x$ defined by
$$
\left\langle\mu^j_x,\varphi\right\rangle:=
\int_{\mathbb R^N} \rho_{j}(x-y)\varphi(y) \, dy
$$
are probability measures, Jensen's inequality yields 
\begin{align}\label{W1cx}
\liminf_{j \to +\infty} 
\int_{A_{\eps,j}}W_1(\nabla u_{\eps,j}(x)) \, dx &=	
\liminf_{j \to +\infty} \int_{A_{\eps,j}}W_1(\nabla (u \star \rho_{j}(x))) \, dx
\nonumber \\
\leq \liminf_{j\to +\infty}
\int_{E_{\eps}}W_1\left(\left\langle\mu^j_x, \nabla u\right\rangle\right) \, dx &\leq
\limsup_{j \to +\infty} \int_{E_{\eps}}
\left\langle\mu^j_x, W_1(\nabla u)\right\rangle \, dx 
\nonumber \\
& = \int_{E_{\eps}}W_1(\nabla u (x)) \, dx 
\end{align}
and this last integral converges, as $\eps \to 0^+$, to 
$\displaystyle \int_{E}W_1(\nabla u (x)) \, dx.$

Similarly, using properties of the convolution, \eqref{growth3} and \eqref{Feps}
we obtain
\begin{equation}\label{W0cx}
\liminf_{\eps \to 0^+}\liminf_{j \to +\infty} 
\int_{\Omega \setminus (E\cup F_\eps)} W_0(\nabla \widetilde{u}_{\eps,j}(x))\, dx \leq
\int_{\Omega \setminus E} W_0(\nabla u(x))\, dx.
\end{equation}
Hence, by \eqref{decomp}, \eqref{W10}, \eqref{nconvper}, \eqref{smallstrip},
\eqref{W1cx} and \eqref{W0cx}, we conclude that
\begin{align*}
& \liminf_{\eps \to 0^+} \liminf_{j\to+\infty} \left [\int_{\Omega}\chi_{\eps}(x)\, W_1(\nabla w_{\eps,j}(x)) + 
(1 -\chi_{\eps}(x))\, W_0(\nabla w_{\eps,j}(x)) \, dx + |D\chi_{\eps}|(\Omega)
\right] \\
& \leq \int_{\Omega}\chi(x)\, W_1(\nabla u(x)) + 
(1 -\chi(x))\, W_0(\nabla u(x)) \, dx + |D\chi|(\Omega) = F\left(\chi,u\right).
\end{align*}
Thus, the conclusion follows by a standard diagonalisation argument.
\end{proof}

\begin{remarks}\label{otherhyp}{\rm 
{\bf 1)} In order to conclude the upper bound, it is possible to consider a constant sequence 
$\chi_\eps = \chi$, the characteristic function of the set $E$, provided we require
more regularity of the function $u$ in a slightly larger set, namely
$u \in W^{1,q}(\Omega \setminus \widehat E;\Rb^d)$, 
where $\widehat E$ is a compact set such that $\widehat E \subset E$. Indeed, under these conditions, $\widetilde{u}_{\varepsilon,j}$ converges, as $j \to + \infty$, to $u$ in  
$W^{1,q}(\Omega \setminus \overline E;\mathbb R^d)$ which allows us to reason as in \eqref{W0cx}. The terms of the energy $F(\chi_\eps,w_{\eps,j})$ have to be adjusted to this choice of admissible sequence $\chi_\eps$ but they can be treated in a similar fashion to \eqref{W10}, \eqref{smallstrip} and \eqref{W1cx}.

{\bf 2)}  
The proof of the upper bound also holds without the assumption of convexity on $W_0$.
In this case, we apply the result of Schmidt twice in 
$\Omega \setminus \overline E$ to obtain sets with smooth boundary 
$\Omega_{2\eps} \subset \subset \Omega_{\eps} \subset \subset \Omega \setminus \overline E$ such that 
$\displaystyle \lim_{\varepsilon \to 0^+}\LL^N(\Omega_{2\eps}) = 
\lim_{\varepsilon \to 0^+}\LL^N(\Omega_{\eps}) =
\LL^N(\Omega \setminus \overline E)$ and 
$\displaystyle \lim_{\varepsilon \to 0^+}P(\Omega_{2\eps};\Rb^N) = 
\lim_{\varepsilon \to 0^+}P(\Omega_{\eps};\Rb^N) =
P(\Omega \setminus \overline E;\Rb^N)$. We denote by $F_\eps$ the layer between $E$ and $\Omega_{\eps}$, we let $L_\eps$ be the inner layer between $\Omega_{\eps}$
and $\Omega_{2\eps}$ and we apply the 
previous slicing argument to $L_\eps$. 
We now denote by 
$u_{\eps,j} := (u \cdot \chi_{E \cup F_\eps \cup L_\eps}) \star \rho_j$.
Since in this case we are partitioning in the complement of $E$, where by hypothesis $u \in W^{1,q}$, we may replace the convolution $\widetilde{u}_{\eps,j}$ with $u$ in the convex combination \eqref{wnova} and this yields an admissible sequence for $\mathcal F(\chi, u;\Omega)$. In this setting, the constant width of each $L^i_{\eps,j}$ is given by 
$\|u_{\eps,j} - u\|_{L^q(L_\eps;\Rb^d)}^{1/q}$ and the optimal transition layer satisfies the equivalent of \eqref{istar} with the $p$ exponents being replaced by $q$. Taking, once again, the sequence $\chi_\eps$ to be the characteristic function of the set $E \cup F_\eps$, the computation of the energy $F(\chi_\eps,w_{\eps,j})$ is bounded above by the terms 
\begin{align*}
\int_EW_1(\nabla u_{\eps,j}(x)) \, dx &+ 
\int_{F_\eps}W_1(\nabla u_{\eps,j}(x)) \, dx +
\int_{L_{\eps}}W_0(\nabla u_{\eps,j}(x)) \, dx \\
&+ \int_{L^{i_*}_{\eps,j}}W_0(\nabla w_{\eps,j}(x)) \, dx + 
\int_{\Omega \setminus E}W_0(\nabla u(x)) \, dx + |D\chi_\eps|(\Omega).
\end{align*}
The first of these can be estimated as before, using the convexity of $W_1$. In the fourth we reason as in \eqref{smallstrip}, using now the growth condition on $W_0$ and taking into account that $L^{i_*}_{\eps,j}$ was chosen using $q$ exponents.
The fifth and sixth terms yield exactly the expressions needed to obtain $F(\chi,u)$ (see \eqref{nconvper}). Finally, to handle the third term, the convexity of $W_0$ is not required since we can use the growth assumption \eqref{growth3} to obtain
\begin{align*}
&\liminf_{\eps \to 0^+}\liminf_{j \to +\infty} 
\int_{L_\eps} W_0(\nabla u_{\eps,j}(x))\, dx \leq
\liminf_{\eps \to 0^+}\liminf_{j \to +\infty}
\int_{L_\eps} C 
\big(1 + \|\nabla u_{\eps,j}(x)\|^q \big)\, dx \\
& \hspace{2.3cm} \leq \liminf_{\eps \to 0^+}\liminf_{j \to +\infty} 
\int_{L_\eps} C 
\big(1 + \|\nabla u_{\eps,j}(x) - \nabla u(x)\|^q  + 
\|\nabla u(x)\|^q\big)\, dx = 0,
\end{align*}
where we have also used the fact that 
$\displaystyle \lim_{\varepsilon \to 0^+}\LL^N(L_\eps) = 0$ and properties of the convolution, and the second term is estimated in a similar fashion using \eqref{growth2}.

In this case it is also possible to work with a fixed $\chi_\eps = \chi$, if
$u \in W^{1,q}(\Omega \setminus \widehat E;\Rb^d)$, 
where $\widehat E$ is a compact set such that $\widehat E \subset E$, which renders the buffer layer $F_\eps$ unnecessary.

{\bf 3)} Taking into account the preceding remark ,
the conclusions of the above theorem remain valid if the convexity of the density $W_0$ is replaced with the weaker requirement that $W_0$ be closed $W^{1,p}$-quasiconvex (cf. \cite[Definition 2]{K}, \cite{P}). Indeed, as mentioned above, the proof of the upper bound is unaffected by this weaker assumption whereas,
since every convex function is closed $W^{1,p}$-quasiconvex, the closed 
$W^{1,p}$-quasiconvexity of $W_0$ entails that of $f(b, \cdot)$, $\forall b \in \{0,1\}$, and so the lower bound is a consequence of Proposition 4.7 in \cite{BZ}. We also point out that under the growth condition \eqref{growth3}, if $W_0$ is closed $W^{1,p}$-quasiconvex then it is also quasiconvex (see \cite[Remark 2.2, Proposition 2.4, Corollary 3.2]{BM} and \cite[Corollary 3.4]{K}).
}

\end{remarks}

\begin{remark}\label{volfrac}{\rm 
The conclusion of Theorem \ref{main} also holds if one prescribes the volume fraction of each phase, provided $u \in W^{1,q}(\Omega \setminus \widehat E;\Rb^d)$, where $\widehat E$ is a compact set such that $\widehat E \subset E$.
Precisely, given $0 < \theta < 1$ and 
$\chi \in BV\left(\Omega;\left\{  0,1\right\}  \right)$ such that
$\displaystyle \frac{1}{\LL^N(\Omega)}\int_\Omega \chi(x) \, dx =\theta$, we define
\begin{align}
\mathcal{F}_{\rm vol}\left(\chi,u;\Omega\right)   &  :=\inf\left\{  
\underset{n\rightarrow +\infty}{\lim\inf}\,
F\left(\chi_{n},u_{n};\Omega\right)  :  u_{n}
\in W^{1,q}\left(\Omega;\mathbb{R}^{d}\right),
\chi_{n}
\in BV\left(\Omega;\left\{  0,1\right\}  \right),  
\right. \nonumber\\
&  \hspace{1cm} \left.  u_{n}\rightharpoonup u\text{ in }
W^{1,p}\left(\Omega;\mathbb{R}^{d}\right),
\chi_{n}\overset{\ast}{\rightharpoonup}\chi\text{ in }
BV\left(\Omega;\left\{  0,1\right\}  \right),  
\frac{1}{\LL^N(\Omega)}\int_\Omega \chi_n(x) \, dx =\theta \right\}.\nonumber
\end{align}
Under the hypotheses of Theorem \ref{main} it follows that
\begin{equation}\nonumber
{\cal F}_{\rm vol}(\chi, u;\Omega) = 
\int_{\Omega}f(\chi(x),\nabla u(x)) \, dx + |D \chi|(\Omega),
\end{equation}
for every $\chi$ characteristic function of an open, connected set of finite perimeter $E \subset \subset \Omega$ such that $E$ and $\Omega \setminus \overline E$ satisfy \eqref{goodset} and $\LL^N(E)= \theta\LL^N(\Omega)$,  and for every 
$u \in W^{1,p}(\Omega;\mathbb R^d)\cap W^{1,q}(\Omega \setminus \widehat E;\mathbb R^d)$.
Indeed, as pointed out in Remarks \ref{otherhyp} {\bf 1)}, under this assumption the constant sequence $\chi_\eps = \chi$ satisfies the desired volume constraint so it is admissible for 
$\mathcal{F}_{\rm vol}\left(\chi,u;\Omega\right)$. Therefore the upper bound follows from the previous proof, whereas the lower bound is obvious.}
\end{remark}

\begin{corollary}\label{Qconv}
Let $\Omega \subset \Rb^N$ be a bounded, open extension domain.
Consider $p, q$ such that \eqref{newpqn} holds,
let $f$ be defined as in 
\eqref{density}, satisfying \eqref{growth2} and \eqref{growth3}.
Assume further that $QW_1$ and $QW_0$, the quasiconvex envelopes of $W_1$ and $W_0$, respectively, are convex functions.
Let $\chi$ be the characteristic function of an open, connected set of finite perimeter
$E \subset \subset \Omega$ such that $E$ and $\Omega \setminus \overline E$ satisfy \eqref{goodset} and let $u \in W^{1,p}(\Omega;\mathbb R^d)$
be such that $u \in W^{1,q}(\Omega \setminus \overline{E};\mathbb R^d)$.
Then, 
\begin{equation}\nonumber
{\cal F}_{\rm loc}(\chi, u;\Omega)= {\cal F}(\chi, u;\Omega) = \int_{\Omega}\chi(x)\, QW_1(\nabla u(x)) + (1 -\chi(x))\,QW_0(\nabla u(x)) \, dx 
+ |D \chi|(\Omega).
\end{equation}
\end{corollary}
\begin{proof}
Let $u_n \in W^{1,q}_{\rm loc}(\Omega;\mathbb R^d)$ and $\chi_n \in BV(\Omega;\{0,1\})$ be such that 
$u_{n}\rightharpoonup u\text{ in }W^{1,p}\left(\Omega;\mathbb{R}^{d}\right)$, 
and $\chi_{n}\overset{\ast}{\rightharpoonup}\chi$ in 
$BV\left(\Omega;\left\{  0,1\right\}\right)$. Then, by the convexity of $QW_1$ and $QW_0$, Ioffe's Theorem \ref{thm2.4ABFvariant} and the lower semicontinuity of the perimeter, we obtain
\begin{align}
& \int_{\Omega}\chi(x)\, QW_1(\nabla u(x)) + (1- \chi(x))\,QW_0(\nabla u(x)) \, dx 
+ |D \chi|(\Omega) \nonumber \\ 
& \leq \liminf_{n\to+\infty}\left(\int_{\Omega}\chi_n(x)\,QW_1(\nabla u_n(x)) 
+ (1- \chi_n(x))\,QW_0(\nabla u_n(x)) \, dx 
+ |D \chi_n|(\Omega)\right) \nonumber \\ 
& \leq \liminf_{n\to+\infty}\left(\int_{\Omega}\chi_n(x)\,W_1(\nabla u_n(x)) 
+ (1- \chi_n(x))\, W_0(\nabla u_n(x)) \, dx 
+ |D \chi_n|(\Omega)\right). \nonumber
\end{align}
Therefore,
$$\int_{\Omega}\chi(x)\, QW_1(\nabla u(x)) + (1- \chi(x))\,QW_0(\nabla u(x)) \, dx 
+ |D \chi|(\Omega) 
\leq {\cal F}_{\rm loc}(\chi, u;\Omega)\leq {\cal F}(\chi,u;\Omega).$$

To prove the reverse inequality, we use the notation established in the proof of Theorem \ref{main} and let $\chi_{\eps}$ and
$w_{\eps,j} \in W^{1,q}(\Omega;\Rb^d)$ be the sequences constructed therein.
Since $\chi_\eps$ takes  only the values 0 or 1 we have
$$\chi_\eps QW_1(\nabla w_{\eps,j}) + (1 - \chi_\eps) QW_0(\nabla w_{\eps,j})
= Qf(\chi_\eps,\nabla w_{\eps,j})$$
where, using the fact that $p \leq q$ and Young's inequality, there exists $C > 0$ such that $f(b,\xi) \leq C(1 + |\xi|^q)$. Therefore,
by standard relaxation results
(cf. \cite{D}, \cite[Theorem 5.4.2]{FL2}) there exists a sequence
$v_{\eps,j,n}\in W^{1,q}(\Omega;\mathbb R^d)$ such that 
$v_{\eps,j,n} \weak w_{\eps,j}$, as $n \to + \infty,$ in 
$W^{1,q}(\Omega;\mathbb R^d)$ and
\begin{align}\label{relax}
&\limsup_{n \to +\infty} \int_{\Omega}f(\chi_\eps(x), \nabla v_{\eps,j,n}(x)) \, dx
+ |D \chi_\eps|(\Omega) \nonumber \\
&=
\int_{\Omega}\chi_\eps(x)\, QW_1(\nabla w_{\eps,j}(x)) + 
(1 -\chi_\eps(x))\, QW_0(\nabla w_{\eps,j}(x)) \, dx + |D\chi_\eps|(\Omega).
\end{align}
As in the previous proof, we estimate the expression in \eqref{relax} by taking into account the definition of $w_{\eps,j}$ in each subset of the decomposition of $\Omega$ given in the proof of Theorem \ref{main}: 
\begin{align*}
& \int_{\Omega}\chi_\eps(x)\, QW_1(\nabla w_{\eps,j}(x)) + 
(1 -\chi_\eps(x))\, QW_0(\nabla w_{\eps,j}(x)) \, dx + |D\chi_\eps|(\Omega) \\ 
& \leq \int_{A_{\eps,j}}QW_1(\nabla u_{\eps,j} (x)) \, dx + 
\int_{\Omega \setminus (E \cup F_\eps)}QW_0(\nabla \widetilde{u}_{\eps,j} (x)) \, dx \\
& + 
\int_{L_{\eps}\cup (E\setminus E_\eps)\cup F_\eps}QW_1(\nabla \widetilde{u}_{\eps,j}(x)) \, dx + 
\int_{L^{i_*}_{\eps,j}} QW_1(\nabla w_{\eps,j}(x))\, dx + |D\chi_\eps|(\Omega).
\end{align*}

Since $QW_1 \leq W_1$, arguments similar to those used to obtain \eqref{smallstrip}
and \eqref{W10} give
$$
\limsup_{j\to+\infty}\int_{L^{i_*}_{\eps,j}} QW_1(\nabla w_{\eps,j}(x))\, dx = 0,
$$
and 
$$
\limsup_{\eps \to 0^+}\limsup_{j\to+\infty}\int_{L_{\eps}\cup (E\setminus E_\eps)\cup F_\eps} QW_1(\nabla \widetilde{u}_{\eps,j}(x))\, dx = 0,
$$
whereas the convexity of $QW_1$, $QW_0$, Jensen's inequality and reasoning as in \eqref{W1cx}, \eqref{W0cx} lead to
$$\liminf_{\eps \to 0^+}\liminf_{j\to+\infty}
\int_{A_{\eps,j}}QW_1(\nabla u_{\eps,j}(x)) \, dx \leq
\int_{E}QW_1(\nabla u(x)) \, dx$$
and 
$$\liminf_{\eps \to 0^+}\liminf_{j\to+\infty}
\int_{\Omega \setminus (E \cup F_\eps)}QW_0(\nabla \widetilde{u}_{\eps,j}(x)) \, dx \leq
\int_{\Omega \setminus E}QW_0(\nabla u(x)) \, dx.$$

Taking into account \eqref{nconvper}, \eqref{relax} and applying once again a standard diagonalisation argument, we obtain sequences 
$u_n \in W^{1,q}(\Omega;\mathbb R^d)$ 
and $\chi_n \in BV(\Omega;\{0,1\})$ 
such that 
$u_{n}\rightharpoonup u\text{ in }W^{1,p}\left(\Omega;\mathbb{R}^{d}\right)$,
$\chi_{n}\overset{\ast}{\rightharpoonup}\chi$ in 
$BV\left(\Omega;\left\{  0,1\right\}\right)$ and
\begin{align*}
& \liminf_{n\to+\infty}\left(\int_{\Omega}\chi_n(x)\,W_1(\nabla u_n(x)) 
+ (1- \chi_n(x))\,W_0(\nabla u_n(x)) \, dx 
+ |D \chi_n|(\Omega)\right)  \\ 
& \leq \int_{\Omega}\chi(x)\, QW_1(\nabla u(x)) + (1- \chi(x))\,QW_0(\nabla u(x)) \, dx + |D \chi|(\Omega).
\end{align*}
Hence,
$$
{\cal F}_{\rm loc}(\chi, u;\Omega)\leq {\cal F}(\chi, u; \Omega) \leq 
\int_{\Omega}\chi(x)\, QW_1(\nabla u(x)) + (1- \chi(x))\,QW_0(\nabla u(x)) \, dx + |D \chi|(\Omega),
$$
and the proof is complete.
\end{proof}
 
The self-contained argument above was presented for the readers' convenience but we observe that if $f$ is as in \eqref{density}, and denoting by $Qf$ its quasiconvex envelope with respect to the second variable (cf. Definition \ref{qcxenv}), this proposition could have been stated and proved in two steps, namely, by showing first that 
 \begin{align*}
 	{\mathcal F}(\chi,u;\Omega)&=\inf \left\{ \liminf_{n \to +\infty}
 	\left[\int_{\Omega}Qf(\chi_n(x), 
 	\nabla u_n(x)) \, dx 
 	+ \left|D \chi_n\right|(\Omega)\right]:  
 	\right.\\ \nonumber
 	& \hspace{0,25cm}\left.	u_n\in W^{1,q}(\Omega;\mathbb R^d), \chi_n\in BV(\Omega;\{0,1\}), 
 	u_n \rightharpoonup u 
 	\hbox{ in }W^{1,p}(\Omega;\mathbb R^d), 
 	\chi_n \weakstar \chi \hbox{ in } BV(\Omega;\{0,1\})\right\},\nonumber
 	\end{align*}
 	as in Lemma \ref{BFMbendingLemma2.3} below, and then by applying Theorem \ref{main}.

\section{Dimension Reduction}

In the sequel we apply the above result to identify the optimal design of plates, in the so-called membranal regime (see e.g. \cite{ld95} and \cite{CZ1} among a wide literature), by means of dimension reduction, in the spirit of the models described in 
\cite{BFF, FF}, which also appear in the context of brutal damage evolution. 
Namely one can deduce, as a rigorous 3$D$-2$D$ $\Gamma$-limit (see \cite{DM} for a detailed treatment of the subject) as 
$\varepsilon \to 0^+$, the optimal design of an elastic membrane $\Omega(\varepsilon):= \omega\times (0,\varepsilon)$, with 
$\omega \subset \mathbb R^2$ a bounded open set with Lipschitz boundary constituted by materials with different hyperelastic responses, i.e., which truly exhibit a gap between the growth and coercivity exponents in the hyperelastic density. 

In the following we adopt the standard scaling (see \cite{CZ1} and the references quoted therein) which maps 
$x\equiv (x_1,x_2,x_3)\in \Omega(\varepsilon) \to 
(x_1, x_2,\frac{1}{\varepsilon}x_3)\in \Omega:=\omega\times (0,1)$, in order to state the problem in a fixed domain (see \eqref{FDR} below). We also denote by $\nabla_\alpha u$ and $D_\alpha \chi$, respectively, the partial derivatives of $u$ and $\chi$ with respect to $x_\alpha\equiv(x_1,x_2)$, while $\nabla_3u$ and $D_3\chi$ represent the derivatives with respect to $x_3$.

In the model under consideration, the sequence 
$\chi_\varepsilon \in BV(\Omega;\{0,1\})$ represents the design regions, whereas $u_\varepsilon \in W^{1,q}(\Omega;\mathbb R^3)$
is the sequence of deformations, which are
clamped at the lateral extremities of the membrane.
Standard arguments in dimension reduction (see e.g. \cite{ld95} and \cite{CZ1}) ensure that energy  bounded sequences (see the term in square brackets of \eqref{FDR}), converge (up to a subsequence), in the relevant topology, to fields $(\chi, u)$ such that $D_3 \chi$ and $\nabla_3u$ are null, thus they can be identified, with an abuse of notation, with fields $(\chi, u)\in BV(\omega;\{0,1\}) \times W^{1,p}(\omega;\mathbb R^3)$.
In what follows we use this notation.

In each of the following subsections we analyse the problem in two different settings, according to the topologies that are considered in the definition of the relaxed energy.

\subsection{The case of $W^{1,q}$ approximating sequences}

\begin{proposition}\label{3D2DOgden}
	Let $\omega\subset \mathbb R^2$ be a bounded, open set
		and  define 
		$\Omega:= \omega\times (0,1)$. Let $1<p \leq q < +\infty$ and
	 $f:\{0,1\} \times \mathbb R^{3 \times 3}\to \mathbb R$ be a continuous function as in \eqref{density}, with $W_i$ as in \eqref{growth2} and \eqref{growth3}
	 with $d = N = 3$.
	Assume also that for every $b \in \{0,1\}$
	\begin{equation}\label{Qffastast}
	Qf(b,\cdot)\hbox{ is convex,} 
	\end{equation}
	where $Q f(b,\cdot)$ denotes the quasiconvex envelope of $f(b,\cdot)$ (see Definition \ref{qcxenv}),  and that there exist $c, c_0 \in \mathbb R^+$ such that
	\begin{equation}\label{Ogdengrowth}
	c|\xi|^p- c_0\leq f(b,\xi), 
	\end{equation}
	for every $ b \in \{0,1\}$ and $\xi \in \mathbb R^{3\times 3}$.

Assume that $\chi$ is the characteristic function of an open, connected set of finite perimeter $E \subset \subset \omega$ such that $E$ and $\omega \setminus \overline E$ have Lipschitz boundary.
Consider a function $u_0 \in W^{1,q}(\Rb^3;\Rb^3)$ such that $u_0(x) = u_0(x_\alpha)$ so that $u_0$ may be identified with a field in $W^{1,q}(\Rb^2;\Rb^3)$. Let 
$u \in u_0 + W^{1,p}_0(\omega;\mathbb R^3)$ 
	be such that $u \in W^{1,q}(\omega \setminus \overline E;\mathbb R^3)$, and let
		\begin{align}\label{FDR}
	\displaystyle{\mathcal F}^{DR}(\chi,u)&:=\inf \left\{ \liminf_{\varepsilon \to 0^+}\left[
	\int_{\Omega}f(\chi_\varepsilon(x), \left(\nabla_\alpha u_\varepsilon(x), \tfrac{1}{\varepsilon}\nabla_3 u_\varepsilon(x))\right) \,dx 
	+ \left|\left(D_\alpha \chi_\varepsilon, \tfrac{1}{\varepsilon}D_3 \chi_{\varepsilon}\right)\right|(\Omega)\right]: \right.\\ 	\nonumber
	\\  
	& \nonumber \hspace{0,3cm}
	u_\varepsilon\in  W^{1,q}(\Omega;\mathbb R^3) {\hbox{ with }u_\eps \equiv u_0 \hbox{ on }\partial \omega \times (0,1) }, \chi_\varepsilon\in BV(\Omega;\{0,1\}), \\ \nonumber
	\\
	& \hspace{0,3cm} \nonumber
\left.	 
	u_\varepsilon \rightharpoonup u 
	\hbox{ in }W^{1,p}(\Omega;\mathbb R^3), 
	\chi_\varepsilon \weakstar \chi \hbox{ in } BV(\Omega;\{0,1\})\right\}.
	\nonumber
	\end{align}
	Then 
	\begin{equation}\label{repFDR}
	{\mathcal F}^{DR}(\chi,u)= \int_\omega Q\widehat f(\chi(x_\alpha),\nabla_\alpha u(x_\alpha)) \,dx_\alpha + |D_\alpha \chi|(\omega),
	\end{equation}
where
	\begin{equation}\label{f0}
	\widehat f(b,\xi_\alpha):=\inf_{\xi_3\in \mathbb R^{3}}f(b,\xi_\alpha,\xi_3),  
	\hbox{ with } b\in \{0,1\}, (\xi_\alpha,\xi_3) \equiv \xi \in \mathbb R^{3\times 3},
	\end{equation}
	and $Q\widehat f(b,\cdot)$ denotes the quasiconvex envelope of $\widehat{f}(b,\cdot)$ with respect to the second variable.
\end{proposition}

We point out that the functional ${\mathcal F}^{DR}$ in \eqref{FDR} is defined in full analogy with $\mathcal F$ in \eqref{introrelaxed}, although it involves an asymptotic process which can be rigorously treated in the framework of $\Gamma$-convergence. 
On the other hand, our proof of the integral representation \eqref{repFDR} is obtained following the same strategy, based on proving a double inequality, adopted at the end of the previous section, and it is self-contained.

Before addressing the proof of Proposition~\ref{3D2DOgden} we start by proving a lemma following the ideas presented in 
\cite[Lemma 2.3]{BFMbending}. 

\begin{lemma}\label{BFMbendingLemma2.3} Under the conditions of Proposition \ref{3D2DOgden} the following holds 
	\begin{align*}
	{\mathcal F}^{DR}(\chi,u)&=\inf \left\{ \liminf_{\varepsilon \to 0^+}
	\left[\int_{\Omega}Qf(\chi_\varepsilon(x), 
	\left(\nabla_\alpha u_\varepsilon(x), \tfrac{1}{\varepsilon}\nabla_3 u_\varepsilon(x))\right) \, dx 
	+ \left|\left(D_\alpha \chi_\varepsilon,\tfrac{1}{\varepsilon} D_3 \chi_\varepsilon\right)\right|(\Omega)\right]:  
	\right.\\ \nonumber
	& \hspace{0,3cm}	u_\varepsilon\in  W^{1,q}(\Omega;\mathbb R^3) \hbox{ with }u_\eps \equiv u_0 \hbox{ on }\partial \omega \times (0,1), \chi_\varepsilon\in BV(\Omega;\{0,1\}), \\
	\nonumber
	& \hspace{0,3cm}\left. u_\varepsilon \rightharpoonup u 
	\hbox{ in }W^{1,p}(\Omega;\mathbb R^3), 
	\chi_\varepsilon \weakstar \chi \hbox{ in } BV(\Omega;\{0,1\})\right\}.\nonumber
	\end{align*}
\end{lemma}
\begin{proof}
	As in \cite[(2.2)]{BFMbending}, we have that
	\begin{equation*}\label{2.2BFM}
	(Qf)_\varepsilon(b,\xi)= Q(f_\varepsilon)(b,\xi), \; \forall b \in \{0,1\},
	\forall \xi \in \mathbb R^{3\times 3},
	\end{equation*} 
	\noindent where the quasiconvex envelopes are taken with respect to the $\xi$ variable and for any function \hfill \linebreak
	$g:\{0,1\}\times \mathbb R^{3\times 3}\to [0,+\infty)$, 
	$$ g_\varepsilon(b,\xi_\alpha,\xi_3):=g\left(b,\xi_\alpha, \tfrac{1}{\varepsilon}\xi_3\right).$$
	In light of \eqref{Qffastast}, $(Qf)_{\varepsilon}$ is convex in the variable 
	$\xi =(\xi_\alpha,\xi_3)$. Similarly, we use the notation $|D_\eps\chi_{\eps}|(\Omega)$ to represent $\left|\left(D_\alpha \chi_\varepsilon,\tfrac{1}{\varepsilon} D_3 \chi_\varepsilon\right)\right|(\Omega) $.
	
	Let ${\mathcal F}^{DR}_{Qf}(\chi,u)$ be defined as ${\mathcal F}^{DR}(\chi,u)$ but replacing $f$ by $Qf$.
	Clearly, since $Qf\leq f$, it follows that ${\mathcal F}^{DR}_{Qf}\leq {\mathcal F}^{DR}$
	so we only need to prove the opposite inequality. 
	To this end, for every $\delta>0$ and every $(\chi, u)$ satisfying the hypotheses of Proposition \ref{3D2DOgden},
    let
	$(\chi_\varepsilon, u_\varepsilon)\in BV(\Omega;\{0,1\})\times  W^{1,q}(\Omega;\mathbb R^3)$ be such that $u_\eps \equiv u_0$ on $\partial \omega \times (0,1)$,
	$u_\varepsilon \weak u$ in $W^{1,p}(\Omega;\mathbb R^3)$, $\chi_\varepsilon \weakstar \chi$ in $BV(\Omega;\{0,1\})$
	and
	$$
	{\mathcal F}^{DR}_{Qf}(\chi,u)\geq \liminf_{\varepsilon \to 0^+}
		\left[\int_{\Omega}Q(f_\eps)
	\left(\chi_\varepsilon(x),\nabla u_\varepsilon(x)\right) \, dx 
	+\left|D_\eps \chi_\varepsilon \right|(\Omega)\right]-\delta.
	$$
Up to the extraction of a subsequence, we may assume that the above lim inf is, in fact, a limit.	

	An application of \cite[Theorem 5.4.2]{FL2}, with a similar argument used to conclude \eqref{relax}, and reasoning as in \cite[Lemma 12]{BFLM} and \cite[Corollary 1.3]{BFMbending}, shows that there exists
	$ u_{\varepsilon, k}\in  W^{1,q}(\Omega;\mathbb R^3)$ such that $ u_{\eps,k} \equiv u_0$ on $\partial \omega \times (0,1)$,
	$
	u_{\varepsilon, k} \rightharpoonup u_{\varepsilon} $ weakly in $W^{1,q}$, as $k \to + \infty$,
	and
	\begin{align}\label{recQf}
	\int_{\Omega}Q(f_\eps)\left(\chi_\varepsilon(x),\nabla u_\varepsilon(x)\right) \, dx +\left|D_\eps \chi_\varepsilon\right|(\Omega)
	= \lim_{k\to +\infty}\int_{\Omega}f_\eps\left(\chi_\varepsilon(x), \nabla u_{\varepsilon,k}(x)\right) \, dx +\left|D_\eps \chi_\varepsilon\right|(\Omega).
	\end{align}
	Thus we can say that
	\begin{equation}\label{doubleenergy}
	{\mathcal F}^{DR}_{Qf}(\chi, u)\geq \lim_{\varepsilon \to 0^+}\lim_{k\to +\infty} \left [
	\int_{\Omega}f_\eps\left(\chi_\varepsilon(x), \nabla u_{\varepsilon,k}(x)\right) \, dx +\left|D_\eps \chi_\varepsilon\right|(\Omega)\right] -\delta, 
	\end{equation}		
	and
	$$\lim_{\varepsilon \to 0^+}\lim_{k\to +\infty}\|u_{\varepsilon, k}-u\|_{L^p(\Omega;\Rb^3)}=0.$$
	The growth from below in \eqref{Ogdengrowth}, the convexity of $|\cdot|^p$ and the fact that
	the weak topology is metrisable on bounded sets, ensure that there exist a diagonal sequence 
	$u_{\varepsilon_k,k}$  and a subsequence $\chi_{\varepsilon_k}$ such that
	$$
	(\chi_{\varepsilon_k},u_{\varepsilon_k,k}) \rightharpoonup (\chi,u) \hbox{ in } BV \hbox{-weak}^\ast \times W^{1,p} \hbox{-weak}, \hbox{ as }
	k \to + \infty,
	$$ 
	the double limit in \eqref{doubleenergy} exists, and thus
	$$
	{\mathcal F}^{DR}_{Qf}(\chi, u) \geq \lim_{k\to +\infty}\left[\int_{\Omega}f_{\eps_k}\left(\chi_{\varepsilon_k}(x), 
	\nabla u_{\varepsilon_k,k}(x)\right) \, dx +
	\left|D_{\varepsilon_k} \chi_{\varepsilon_k}\right|(\Omega)\right] -\delta,
	$$
	which, in turn, implies that
	\begin{equation}\label{calFdelta}
	{\mathcal F}^{DR}_{Qf}(\chi, u) \geq {\mathcal F}^{DR}(\chi, u) -\delta.
\end{equation}
	It suffices to let $\delta \to 0^+$ to conclude the proof.
\end{proof}

\begin{lemma}
	\label{QCX0Properties}
	Assume that $f$ is as in Proposition \ref{3D2DOgden}, and $Qf$, its quasiconvex envelope with respect to the second variable, satisfies \eqref{Qffastast}
	Then, for every $b \in \{0,1\}$,
	\begin{equation}
	\label{Qf0=}
	\widehat {Qf}(b,\cdot)=Q \widehat f(b,\cdot),
	\end{equation} 
		where, for each function $g:\{0,1\}\times\mathbb R^{3 \times 3} \to [0,+\infty)$, $\widehat g:\{0,1\}\times\mathbb R^{3 \times 2}\to [0,+\infty)$ is defined as in \eqref{f0}.
\end{lemma}
\begin{proof}
Rewriting, as in \eqref{f0}, $\xi$ as $(\xi_\alpha, \xi_3) \in \mathbb R^{3 \times 3}$, we observe that $\widehat f(b, \xi_\alpha)\leq f(b,\xi_\alpha, \xi_3)$ for every $(b,\xi_\alpha, \xi_3) \in \{0,1\}\times \mathbb R^{3 \times 3}$, thus 
	\begin{equation}\label{Qbarf}
	Q \widehat f(b,\xi_\alpha) \leq Q f(b,\xi_\alpha,\xi_3)
	\end{equation} for every $(b,\xi_\alpha, \xi_3) \in \{0,1\}\times \mathbb R^{3 \times 3}$, where, with an abuse of notation, $\widehat f$ and $Q\widehat f(b,\cdot)$ are considered as defined in $\{0,1\}\times \mathbb R^{3\times 3}$, assuming that they are independent of $\xi_3$, the quasiconvex envelope on the right hand side of \eqref{Qbarf} is taken with respect to the variable $(\xi_\alpha, \xi_3)\in \mathbb R^{3\times 3}$, and we are taking into account, as in \cite[Proposition 6]{ld95}, that $Q\widehat f(b,\cdot)$ is quasiconvex as a function of $(\xi_\alpha, \xi_3)$.
	Then, applying \eqref{f0} to both sides of \eqref{Qbarf} we have 
	$$
	Q \widehat f(b,\xi_\alpha)=\widehat{Q \widehat f}(b,\xi_\alpha) \leq \widehat{Q f}(b,\xi_\alpha),
	$$
	for every $(b,\xi_\alpha) \in \{0,1\}\times \mathbb R^{3 \times 2}$, which proves one inequality.
	
	For what concerns the reverse inequality, since $Qf(b,\xi) \leq f(b,\xi)$ for every $(b,\xi)\in \{0,1\} \times \mathbb R^{3\times 3}$, we have
	$$\widehat{Qf}(b,\xi_\alpha) \leq \widehat f(b,\xi_\alpha),$$
	for every $(b,\xi_\alpha)\in \{0,1\} \times \mathbb R^{3\times 2}$. On the other hand, it is easily seen (cf. also \cite[(5.10)]{BZ}) that \eqref{Qffastast} entails the convexity of $\widehat{Qf}$ with respect to the variable $\xi_\alpha$, thus  $\widehat{Qf} $ is quasiconvex with respect to $\xi_\alpha$, hence $$\widehat{Qf}(b,\xi_\alpha)=Q(\widehat{Qf})(b,\xi_\alpha)\leq Q\widehat f(b,\xi_\alpha)$$ for every $(b,\xi_\alpha) \in \{0,1\}\times \mathbb R^{3\times 2}$ which concludes the proof.
	
\end{proof}

\begin{proof}
	[Proof of Proposition \ref{3D2DOgden}]
	The proof of \eqref{repFDR} is obtained by showing a double inequality.
	We use Lemma \ref{BFMbendingLemma2.3} and replace $f$ by $Qf$. 
	We also point out that the hypotheses placed on $E$ and $\omega \setminus \overline E$ imply, in particular, that $E$ and $\omega \setminus \overline E$ satisfy \eqref{goodset} with $N=2$. 
	
	For what concerns the lower bound, we first observe that
	by \eqref{Ogdengrowth} we have
	\begin{equation}
	\nonumber
	c|\xi_\alpha|^p- c_0 \leq Q\widehat f(b,\xi_\alpha),
	\end{equation}
	for every $(b,\xi_\alpha) \in \{0,1\}\times \mathbb R^{3 \times 2}$, and 
	$Q\widehat{W_0}$ and $Q\widehat{W_1}$ satisfy \eqref{growth2} and \eqref{growth3}, respectively.	On the other hand, we recall that, 
	since $b$ takes only the values 0 and 1 and the quasiconvex envelope is taken with respect to the variable $\xi_\alpha$,
	\begin{equation}\label{QfQWi}Q\widehat f (b,\xi_\alpha)= bQ\widehat{ W_1}(\xi_\alpha)+(1-b)Q\widehat{W_0}(\xi_\alpha),
	\end{equation} 
	for every $(b,\xi_\alpha)\in\{0,1\}\times \mathbb R^{3 \times 2}$.
	
\noindent	Moreover the functional 
	$$\int_\Omega Q\widehat f(\chi(x_\alpha, x_3),\nabla_\alpha u (x_\alpha, x_3)) \, dx$$ 
	is lower semicontinuous with respect to $BV$-weak $\ast \times W^{1,p}$ -weak convergence by Theorem \ref{thm2.4ABFvariant}. Indeed by Lemma \ref{QCX0Properties}, $Q\widehat f(b,\xi_\alpha)=\widehat {Qf}(b,\xi_\alpha)$, thus, by \eqref{Qffastast} and \eqref{f0}, it is convex in the second variable. Then, the superadditivity of the limit inf, the fact that $\left|\left(D_\alpha \chi_{\varepsilon}, \tfrac{1}{\varepsilon}D_3 \chi_\varepsilon\right)\right|(\Omega) \geq |D_\alpha \chi_\varepsilon|(\Omega)$ and the lower semicontinuity of the total variation, entail that for any admissible pair $(\chi_\eps, u_\eps)$
\begin{align*}
&\int_\omega Q{\widehat f}(\chi(x_\alpha),\nabla_\alpha u(x_\alpha))\,dx_\alpha + |D_\alpha \chi|(\omega) \leq 
	\liminf_{\varepsilon \to 0^+} \left[\int_{\Omega}Q {\widehat f}(\chi_\varepsilon(x_\alpha, x_3),\nabla_\alpha u_\varepsilon (x_\alpha, x_3))\,dx +|D_\alpha \chi_\varepsilon|(\Omega)\right] \\
	& \hspace{4cm} \leq 	\liminf_{\varepsilon \to 0^+} \left[\int_{\Omega}
	Qf\left(\chi_\varepsilon(x),(\nabla_\alpha u_\varepsilon (x), \tfrac{1}{\eps}\nabla_3 u_\eps(x))\right)\,dx +\left|\left(D_\alpha \chi_\varepsilon,\tfrac{1}{\eps}D_3\chi_{\eps}\right)\right|(\Omega)\right]
\end{align*}
	so that 
	$${\mathcal F}^{DR}(\chi, u) \geq \int_\omega Q{\widehat f}(\chi(x_\alpha),\nabla_\alpha u(x_\alpha))\,dx_\alpha + |D_\alpha \chi|(\omega).$$

In order to prove the upper bound,
We use a two-dimensional version of the proof and the notations of Theorem \ref{main} and Remarks \ref{otherhyp} {\bf 2)}. 
	
Let $\chi \in BV(\omega;\{0,1\})$ be the characteristic function of an open, connected set of finite perimeter $E \subset \subset \omega$ such that \eqref{goodset} holds for $E$ and $\omega \setminus \overline E$,
i.e. $\mathcal H^1(\partial E)= P(E;\omega)= P(E;\mathbb R^2)$, 
$\mathcal H^1(\partial(\omega \setminus \overline E))= P(\omega \setminus \overline E;\mathbb R^2)$,
and let 
$u \in u_0 + W^{1,p}_0(\omega;\mathbb R^3)\cap W^{1,q}(\omega\setminus \overline E;\mathbb R^3)$. 
Consider $\chi_{\eps}(x_\alpha)$, the characteristic function of the set $E \cup F_\eps$, 
and let $w_{\eps,j}(x_\alpha)$ be defined as 
\begin{equation}\label{wdimred}
w_{\eps,j}(x_\alpha):= \varphi_{\eps,j}(x_\alpha) u_{\eps,j}(x_\alpha) +(1-\varphi_{\eps,j}(x_\alpha)){u}(x_\alpha),
\end{equation} 
where $\varphi_{\eps,j}$ is the two-dimensional version of the sequence of cut-off functions considered in the proof of Theorem \ref{main}. 
Given $\psi \in  W^{1,p}_0(\omega;\mathbb R^3)\cap W^{1,q}_0(\omega \setminus \overline E;\mathbb R^3)$, we regularise $\psi$ 
in the same way as in Remarks \ref{otherhyp} {\bf 2)}, that is, given the usual sequence of standard mollifiers
$\{\rho_j\}_{j \in \Nb}$, we consider
$\psi_{\eps,j}(x_\alpha) = ((\psi \cdot \chi_{E \cup F_\eps \cup L_\eps}) 
\star \rho_j)(x_\alpha)$ 
and, in a similar fashion to \eqref{wdimred}, we let 
\begin{equation}\label{eta}
\eta_{\eps,j}(x_\alpha):= \varphi_{\eps,j}(x_\alpha) \psi_{\eps,j}(x_\alpha) +(1-\varphi_{\eps,j}(x_\alpha)) \psi(x_\alpha),
\; \; x_\alpha\in \omega.
\end{equation}
We now define
$$v_{\varepsilon,j}(x):= w_{\eps,j}(x_\alpha)+ \varepsilon \, x_3\,\eta_{\eps,j}(x_\alpha),$$
and, by abuse of notation, consider 
$\chi_\eps(x) = \chi_\eps(x_\alpha)$.
Clearly $\{v_{\varepsilon,j}\}$ and $\{\chi_\eps\}$ are admissible for 
${\mathcal F}^{DR}(\chi, u)$ so we obtain, using Lemma \ref{BFMbendingLemma2.3},
\begin{align*}
{\mathcal F}^{DR}(\chi, u) &\leq \liminf_{\eps \to 0^+}\liminf_{j\to+\infty}
\left[\int_{\Omega}Qf\big(\chi_\eps(x),(\nabla_\alpha v_{\varepsilon,j}(x),\tfrac{1}{\eps}\nabla_3v_{\varepsilon,j}(x))\big)\, dx +
|(D_\alpha\chi_\eps,\tfrac{1}{\eps}D_3\chi_\eps)|(\Omega)\right]\\
& = \liminf_{\eps \to 0^+}\liminf_{j\to+\infty}
\left[\int_{\Omega}Qf\big(\chi_\eps(x),(\nabla_\alpha w_{\varepsilon,j}(x_\alpha)+\eps x_3\nabla_\alpha \eta_{\varepsilon,j}(x_\alpha),\eta_{\varepsilon,j}(x_\alpha))\big)\, dx +
|D_\alpha\chi_\eps|(\Omega)\right]\\
&\leq  \int_{\Omega}Qf\big(\chi(x),(\nabla_\alpha u(x_\alpha),\psi(x_\alpha))\big)\, dx + |D\chi|(\Omega)\\
&= \int_{\omega}Qf\big(\chi(x_\alpha),(\nabla_\alpha u(x_\alpha),\psi(x_\alpha))\big)\, dx_\alpha + |D_\alpha\chi|(\omega),
\end{align*}
where the inequality on the third line is proved following the estimates provided in the proof of Theorem \ref{main} and in Remarks \ref{otherhyp}, and also using the $q$-Lipschitz continuity of $W_0$ and the $p$-Lipschitz continuity of $W_1$. Hence, given the arbitrariness of $\psi$, we conclude that
$${\mathcal F}^{DR}(\chi, u)\leq |D_\alpha \chi|(\omega)+\inf_{\psi \in W^{1,p}_0(\omega;\mathbb R^3)\cap W^{1,q}_0(\omega\setminus \overline E;\mathbb R^3)} \int_\omega Qf\big(\chi(x_\alpha),(\nabla_\alpha u(x_\alpha), \psi(x_\alpha))\big)\,dx_\alpha.$$
On the other hand, the growth conditions \eqref{growth2}, \eqref{growth3} and a density argument show that
\begin{align*}
&\inf_{\psi \in W^{1,p}_0(\omega;\mathbb R^3)\cap W^{1,q}_0(\omega\setminus \overline E;\mathbb R^3)} \int_\omega Qf\big(\chi(x_\alpha),(\nabla_\alpha u(x_\alpha), \psi(x_\alpha))\big)\,dx_\alpha \\
& \hspace{4cm}=
\inf_{\psi \in L^p(\omega;\mathbb R^3)\cap L^q(\omega\setminus \overline E;\mathbb R^3)} \int_\omega Qf\big(\chi(x_\alpha),(\nabla_\alpha u(x_\alpha), \psi(x_\alpha))\big)\,dx_\alpha.
\end{align*}

	Recalling the continuity and the coercivity of $Q f(b,\cdot)$, as in \eqref{Ogdengrowth}, and using Lemma \ref{QCX0Properties}, \eqref{QfQWi},
	and the measurability criterion which provides the existence of $\bar\psi \in L^p(\omega;\mathbb R^3)\cap L^q(\omega \setminus \overline E;\mathbb R^3)$ such that
	$$
	Q\widehat f(\chi(x_\alpha),\nabla_\alpha u(x_\alpha))=\widehat{Qf}(\chi(x_\alpha),\nabla_\alpha u(x_\alpha))=Qf\big(\chi(x_\alpha),
	(\nabla_\alpha u(x_\alpha),\bar{\psi}(x_\alpha))\big),
	$$
	it follows that 
	$$
	{\mathcal F}^{DR}(\chi, u)\leq  |D_\alpha \chi|(\omega)+ \int_\omega Q\widehat f(\chi(x_\alpha),\nabla_\alpha u(x_\alpha)) \,dx_\alpha,
	$$
	which completes the proof. 
\end{proof}

In order to deal with optimal design problems where the volume fraction of each phase is prescribed, i.e. as in \eqref{vol}, 
it is easily seen that the constraint $\displaystyle \frac{1}{\LL^N(\Omega)}\int_{\Omega}\chi(x) \, dx = \theta, 
\theta \in (0,1)$, does not affect at all our proof, if we insert it in the form of a Lagrange multiplier into the model, that is, we can add 
$\displaystyle \kappa \int_{\Omega}\chi(x) \, dx, \, \kappa >0$, to the functional $F$ since this is a linear term. 

On the other  hand this choice allows us to interpret the representation result in Proposition \ref{3D2DOgden}, in the light of ``brutal damage evolution models for thin films'' as proposed in \eqref{Fdam}, where, in fact, the linear term describes a dissipation energy. 

Another possibility to deal with the volume constraint is to argue as in Remark \ref{volfrac}.

\subsection{The case of $W^{1,p(x)}$ approximating sequences}

In the sequel we present a dimension reduction result in the framework of Sobolev spaces with piecewise constant exponents (cf. \cite{CM} and \cite{CF, DHHR} for more details on variable Lebesgue spaces).
We recall that for every bounded function $p:\Omega \to [1,+\infty)$, the Lebesgue and Sobolev spaces with variable exponents are defined as
\begin{align}\nonumber
L^{p(x)} (A;\mathbb R^N) &:=\left\{u: A \to \mathbb R^N:
\int_A|u|^{p(x)} dx <+\infty\right\},\\
W^{1,p(x)} (A;\mathbb R^N)&:=\left\{
u \in L^{p(x)} (A;\mathbb R^N) :\nabla u \in L^{p(x)}(A;\mathbb R^{n\times N})\right\},\nonumber \\
W^{1,p(x)}_0 (A;\mathbb R^N)&:=\left\{
u \in W^{1,1}_0 (A;\mathbb R^N) :\nabla u \in L^{p(x)}(A;\mathbb R^{n\times N})\right\},
\label{LpxW1px}
\end{align}
for every open subset $A$ of $\Omega$.

\begin{proposition}\label{3D2DOgdenCM}
	Let $\omega\subset \mathbb R^2$ be a bounded, open set with Lipschitz boundary
	and  define 
	$\Omega:= \omega\times (0,1)$. Let $1<p \leq q < +\infty$ and
	$f:\{0,1\} \times \mathbb R^{3 \times 3}\to \mathbb R$ be a continuous function as in \eqref{density}, with $W_i$ such that \begin{align}\label{q} 
	&\frac{1}{C}|\xi|^q \leq W_0(\xi)\leq C(|\xi|^q+1),\\
	&\frac{1}{C'}|\xi|^p \leq W_1(\xi)\leq C'(|\xi|^p+1)  \; \hbox{ for every }\xi \in \mathbb R^{3\times 3}. \label{p}
	\end{align}
	 for suitable positive constants $C$ and $C'$.
	 Assume also that \eqref{Qffastast} holds, for every $b \in \{0,1\}$.

Let $\chi\in BV(\omega;\{0,1\})$ be the characteristic function of an open, connected set with Lipschitz boundary $E\subset \subset \omega$ such that $\mathcal L^2(\partial E)=0$, $\omega \setminus \overline E$ has Lipschitz boundary and, with an abuse of notation, assume that $\chi\in BV(\Omega;\{0,1\})$ by setting $\chi(x) = \chi(x_\alpha)$. Let 
$u_0 \in W^{1,q}(\Rb^3;\Rb^3)$ be such that $u_0(x) = u_0(x_\alpha)$ so that $u_0$ may be identified with a field in $W^{1,q}(\Rb^2;\Rb^3)$, and assume that	
$u \in u_0 + W^{1,\chi(x)p+ (1-\chi(x))q}_0(\omega;\mathbb R^3)$, where 
this space is to be understood in the sense of \eqref{LpxW1px}. 

Denote by $\mathcal A$ the subset of $BV(\Omega;\{0,1\})\times W^{1,p}(\Omega;\mathbb R^3)$ composed of sequences $\{(\chi_\varepsilon, u_\varepsilon)\}$ converging strongly in $L^1(\Omega;\{0,1\})\times L^1(\Omega;\mathbb R^3)$ to $(\chi,u)$ and such that 
$$
\liminf_{\varepsilon \to 0^+}\left[\int_{\Omega}
f_\eps(\chi_\varepsilon(x), \nabla u_\varepsilon(x)) \,dx 
+ \left|D_\eps \chi_\varepsilon \right|(\Omega)\right]<+\infty.$$
	Let
	\begin{align}\label{FDRCM}
	\overline {\mathcal F}(\chi,u):=\inf \left\{ \liminf_{\varepsilon \to 0^+}
	\left[\int_{\Omega}
	f_\eps(\chi_\varepsilon(x), \nabla u_\varepsilon(x)) \,dx 
	+ \left|D_\eps \chi_\varepsilon \right|(\Omega)\right]: 
	(\chi_\varepsilon, u_\varepsilon)\in {\mathcal A},	
	u_\eps \equiv u_0 \hbox{ on } \partial \omega \times (0,1)
	\right\}.
	\end{align}
	Then 
	\begin{equation}\label{repFCM}
	\overline{\mathcal F}(\chi,u) = \int_\omega Q\widehat f(\chi(x_\alpha),\nabla_\alpha u(x_\alpha)) \,dx_\alpha + 
	|D_\alpha \chi|(\omega),
	\end{equation}
	where $	\widehat f(b,\xi_\alpha)$ is given by \eqref{f0},
	and, as above, $Q\widehat f(b,\cdot)$ denotes the quasiconvex envelope of $\widehat{f}(b,\cdot)$ with respect to the second variable.
\end{proposition}

We point out that the functional $\overline{\mathcal F}$ in \eqref{FDRCM} is defined by means of an asymptotic process which can be rigorously treated in the framework of $\Gamma$-convergence in Sobolev spaces with variable exponents, making use of the results proven in \cite{CM}. 
We observe that due to \eqref{q} and \eqref{p}, the strong  $L^1$ convergence of the admissible sequences $\{(\chi_\varepsilon, u_\varepsilon)\}$ in $\mathcal A$  towards $(\chi,u)$ as in the statement, can be replaced by $BV$-weak*$\times W^{1,p}$-weak convergence of sequences for which the functional in \eqref{FDRCM} is finite.

In the proof of Proposition~\ref{3D2DOgdenCM} we will make use of an analogue of Lemma \ref{BFMbendingLemma2.3}, which allows us to replace $f_\eps$ by $Q(f_\eps)$ in \eqref{FDRCM}. Although the sequences and the convergences are taken in a different setting, we will see that the result remains true.

In what follows it will also be useful to keep in mind that \eqref{Qf0=} holds.

\begin{proof}
	[Proof of Proposition \ref{3D2DOgdenCM}]
	The proof of \eqref{repFCM} is obtained by showing a double inequality.
	For what concerns the lower bound, it suffices to observe that 
	$$Q\widehat f(b, \xi_\alpha) \leq Qf(b,\xi_\alpha,\xi_3), \;\forall b \in \{0,1\}, \forall \xi = (\xi_\alpha, \xi_3) \in \Rb^{3 \times 3},$$ and 
	$Q\widehat{W_0}$ and $Q\widehat{W_1}$ satisfy \eqref{q}, \eqref{p} and, in addition, \eqref{QfQWi} holds
	for every $(b,\xi_\alpha)\in\{0,1\}\times \mathbb R^{3 \times 2}$.
	
	Moreover, the same arguments adopted in the proof of Proposition \ref{3D2DOgden} provide the lower semicontinuity of the functional 
	$$\int_\Omega Q\widehat f(\chi(x_\alpha, x_3),\nabla_\alpha u (x_\alpha, x_3)) \, dx$$ 
	 with respect to  $L^1$ strong convergence. Indeed, due to the convexity of $Q\widehat f$ in the second variable, it suffices, once again, to invoke  Theorem \ref{thm2.4ABFvariant}.
	This, in addition to the lower semicontinuity of the total variation and the superadditivity of the limit inf, yields the lower bound.

In order to prove the upper bound, we observe that 
$\overline{\mathcal F}$ in \eqref{FDRCM} is bounded from above by the functional defined below, where in the admissible sequences we work with a fixed 
$\chi \in BV(\omega;\{0,1\})$, the characteristic function of a set 
$E \subset \subset \omega$ as in the statement,
\begin{equation}\label{FDRCMchi}
\overline{\mathcal F}_{\chi}(\chi,u):=\inf\left\{\liminf_{\varepsilon \to 0^+}
\left[\int_{\Omega}f_\eps(\chi(x),\nabla u_\varepsilon(x)) \,dx 
+ \left|D_\eps \chi\right|(\Omega)\right]: 
(\chi, u_\varepsilon)\in {\mathcal A}, u_\eps \equiv u_0 \hbox{ on } \partial \omega \times (0,1)\right\}.
\end{equation}
Arguing as  
in \cite[Corollary 1.3]{BFMbending}, we observe that the relaxation procedure leads to the same result both when considering sequences clamped on the lateral boundary and without prescribed lateral boundary datum. On the other hand, we show that
replacing $f_\eps$ by $Q(f_\eps)$ leads to the same functional. Indeed, denoting by 
$\overline{\mathcal F}_{\chi, Q(f_\eps)}$ the functional as in \eqref{FDRCMchi} with $f_\eps$ replaced by $Q(f_\eps)$, it is clear that
$$\overline{\mathcal F}_{\chi}(\chi, u)
 \geq \overline{\mathcal F}_{\chi, Q(f_\eps)}(\chi, u).$$
On the other hand, in view of \cite[Corollary 6.3]{CM}, since both our sets $E$ and $\omega \setminus \overline E$ have Lipschitz boundary, the arguments used
in Lemma \ref{BFMbendingLemma2.3} to conclude  \eqref{recQf} and \eqref{doubleenergy}
ensure the existence of an admissible sequence $u_{\eps,k} \in W^{1,p}(\Omega;\Rb^3)$ satisfying the lateral boundary condition and such that, for every $\delta >0$,
\begin{equation*}
 	\overline{\mathcal F}_{\chi, Q(f_\eps)}(\chi, u)\geq \lim_{\varepsilon \to 0^+}\lim_{k\to +\infty} \left [
 	\int_{\Omega}f_\eps\left(\chi(x),\nabla u_{\varepsilon,k}(x)\right) \, dx +\left|D_\eps \chi\right|(\Omega)\right] -\delta, 
 \end{equation*}		
 and
 $$\lim_{\varepsilon \to 0^+}\lim_{k\to +\infty}\|u_{\varepsilon, k}-u\|_{L^1(\Omega;\Rb^3)}=0.$$
 Then the same diagonalization argument as in Lemma \ref{BFMbendingLemma2.3}	allows us to conclude the equality 
 $$\overline{\mathcal F}_{\chi}(\chi, u)
 =\overline{\mathcal F}_{\chi, Q(f_\eps)}(\chi, u).$$
 
	Next, we will reason as in \cite{ld95}.  
	 Let
	$u \in u_0 +
	W_0^{1,{\chi p+ (1-\chi)q}}(\omega;\mathbb R^3)$.
	Given $\psi \in W^{1,p}_0(E)\cap W^{1,q}_0(\omega \setminus \overline E)$
define
	\begin{equation*}
	\eta_{\eps}(x_\alpha):= u(x_\alpha)+ \varepsilon x_3 \psi(x_\alpha), 
	\; \; x_\alpha\in \omega.
\end{equation*}
	Thus, by the chain of inequalities
$$
	\overline{\mathcal F}(\chi,u) \leq \overline{\mathcal F}_{\chi}(\chi, u)
	 =\overline{\mathcal F}_{\chi, Q(f_\eps)}(\chi, u),
	$$ 
	we have 
	\begin{align*}
		\overline{\mathcal F}(\chi,u) &\leq \liminf_{\eps \to 0^+}
		\left[\int_{\Omega}Qf\big(\chi(x),(\nabla_\alpha \eta_{\varepsilon}(x),\tfrac{1}{\eps}\nabla_3\eta _{\varepsilon}(x))\big)\, dx +
		|(D_\alpha\chi,\tfrac{1}{\eps}D_3\chi)|(\Omega)\right]\\
		& = \liminf_{\eps \to 0^+}
		\left[\int_{\Omega}Qf\big(\chi(x),(\nabla_\alpha u(x_\alpha)+\eps x_3\nabla_\alpha \psi(x_\alpha),\psi(x_\alpha))\big)\, dx +
		|D_\alpha\chi|(\Omega)\right]\\
		&\leq  \int_{\Omega}Qf\big(\chi(x),(\nabla_\alpha u(x_\alpha),\psi(x_\alpha))\big)\, dx + |D\chi|(\Omega)\\
		&= \int_{\omega}Qf\big(\chi(x_\alpha),(\nabla_\alpha u(x_\alpha),\psi(x_\alpha))\big)\, dx_\alpha + |D_\alpha\chi|(\omega),
		\end{align*}
	where the inequality on the third line is proved exploiting the $q$-Lipschitz continuity of $W_0$ and the $p$-Lipschitz continuity of $W_1$. Hence, given the arbitrariness of $\psi \in W^{1,p}_0(E)\cap W^{1,q}_0(\omega \setminus \overline E)$, we conclude that
	$$\overline{\mathcal F}(\chi,u)\leq 
	|D_\alpha \chi|(\omega)+\inf_{\psi \in W^{1,p}_0(E)\cap W^{1,q}_0(\omega \setminus \overline E)} \int_\omega Qf\big(\chi(x_\alpha),(\nabla_\alpha u(x_\alpha), \psi(x_\alpha))\big)\,dx_\alpha.$$
	Thus, the growth conditions \eqref{q}, \eqref{p} and a density argument show that
		\begin{align*}
		&\inf_{\psi \in W^{1,p}_0(E;\mathbb R^3)\cap W^{1,q}_0(\omega\setminus \overline E;\mathbb R^3)} \int_\omega Qf\big(\chi(x_\alpha),(\nabla_\alpha u(x_\alpha), \psi(x_\alpha))\big)\,dx_\alpha \\
		& \hspace{4cm}=
		\inf_{\psi \in L^p(\omega;\mathbb R^3)\cap L^q(\omega\setminus \overline E;\mathbb R^3)} \int_\omega Qf\big(\chi(x_\alpha),(\nabla_\alpha u(x_\alpha), \psi(x_\alpha))\big)\,dx_\alpha.
		\end{align*}

	Recalling the continuity and the coercivity of $Q f(b,\cdot)$, in each of the two terms present in \eqref{density}, and using Lemma \ref{QCX0Properties}, and equation \eqref{QfQWi},
	and the measurability criterion which provides the existence of $\bar\psi \in L^p(\omega;\mathbb R^3)\cap L^q(\omega \setminus \overline E;\mathbb R^3)$ such that
	$$
	Q\widehat f(\chi(x_\alpha),\nabla_\alpha u(x_\alpha))=\widehat{Qf}(\chi(x_\alpha),\nabla_\alpha u(x_\alpha))=Qf\big(\chi(x_\alpha),
	(\nabla_\alpha u(x_\alpha),\bar{\psi}(x_\alpha))\big),
	$$
	it follows that 
	$$
	\overline{\mathcal F}(\chi,u)\leq  |D_\alpha \chi|(\omega)+ \int_\omega Q\widehat f(\chi(x_\alpha),\nabla_\alpha u(x_\alpha)) \,dx_\alpha,
	$$
	which completes the proof. 
\end{proof}

Neglecting lateral boundary conditions and compactness arguments, we observe that the above result holds true for the weak convergence in 
$W^{\chi(x) p+(1-\chi(x)) q}_{\rm loc}(\Omega;\mathbb R^3)$, relaxing the conditions on $E$, for example the connectedness, requiring simply that $E$ and $\omega\setminus \overline E$ provide a partition of $\omega$ in the sense of \cite[Theorem 6.1]{CM}.
In the same spirit of {\bf 3)} Remarks \ref{otherhyp}, we point out that the above result remains valid if we replace \eqref{Qffastast} with the assumption that $Q\widehat f$ is closed $W^{1,p}$-quasiconvex. Indeed, the proof of the upper bound required no convexity hypothesis on $f$ and to conclude the lower bound it would suffice to apply \cite[Proposition 4.7]{BZ} to $Q\widehat f$.

\medskip

\noindent\textbf{Acknowledgements}.
The incentive to pursue the topic of this paper arose during the International Conference on Elliptic and Parabolic Problems 2019. The authors would like to thank G. Dal Maso and S. Kr\"omer for having proposed it and for discussions on this subject, as well as the anonymous referees for their careful reading of the manuscript.
We also thank CMAFcIO at the Universidade de 
Lisboa and Dipartimento di Ingegneria Industriale at the Universit\`a degli Studi di Salerno (which the second author was affiliated with during the course of this research), where this 
research was carried out, and gratefully acknowledge the support
of INdAM GNAMPA, Programma Professori Visitatori 2019. The research of ACB was partially supported by National Funding from FCT -
Funda\c c\~ao para a Ci\^encia e a Tecnologia through project UIDB/04561/2020.
EZ is a member of INdAM GNAMPA, whose support is gratefully acknowledged.

\end{document}